\documentclass[10pt]{amsart}

\usepackage{amssymb}
\usepackage{amsthm}

\usepackage[figuresright]{rotating}

\usepackage{amsmath}

\newtheorem{theorem}{Theorem}[section]
\newtheorem{lemma}[theorem]{Lemma}

\newtheorem{proposition}[theorem]{Proposition}

\newtheorem{main-theorem}[theorem]{Theorem}

\theoremstyle{definition}

\newtheorem{definition}[theorem]{Definition}

\newtheorem{example}[theorem]{Example}

\usepackage[normalem]{ulem} 

\usepackage{dsfont}

\usepackage{empheq} 

\renewcommand{\Im}{\operatorname{Im}}

\renewcommand{\mod}{\operatorname{mod}}

\newcommand{\Ext}{\operatorname{Ext}}

\newcommand{\soc}{\operatorname{soc}}
\newcommand{\op}{\operatorname{op}}

\newcommand{\rad}{\operatorname{rad}}

\newcommand{\Ker}{\operatorname{Ker}}

\newcommand{\cA}{\mathcal{A}}
\newcommand{\cB}{\mathcal{B}}

\newcommand{\cO}{\mathcal{O}}

\newcommand{\bN}{\mathbb{N}}
\newcommand{\bP}{\mathbb{P}}

\newcommand{\bR}{\mathbb{R}}

\newcommand{\ba}{\bar{\alpha}}
\newcommand{\La}{\Lambda}
\newcommand{\vf}{\varphi}
\newcommand{\overbar}[1]{\mkern 5mu\overline{\mkern-5mu#1\mkern-5mu}\mkern 5mu}

\usepackage[matrix,arrow,curve,frame,arc]{xy}

\usepackage{pgf}
\usepackage{tikz}

\usetikzlibrary{arrows,shapes,automata,backgrounds,petri,calc}

\usepackage{enumerate}

\begin{document}


\title{Socle equivalences of weighted surface algebras}

{\def\thefootnote{}
\footnote{The first, third and fourth named authors 
were supported by the research grant
DEC-2011/02/A/ST1/00216 of the National Science Center Poland.
The fifth named author was 
supported by
JSPS KAKENHI Grant Number JP19K03417.}
}

\author[J. Bia\l kowski]{Jerzy Bia\l kowski}

\author[K. Erdmann]{Karin Erdmann}

\author[A. Hajduk]{Adam Hajduk}

\author[A. Skowro\'nski]{Andrzej Skowro\'nski}

\author[K.~Yamagata]{Kunio~Yamagata}

\address[Jerzy Bia\l kowski, Adam Hajduk, A. Skowro\'nski]{Faculty of Mathematics and Computer Science,
   Nicolaus Copernicus University,
   Chopina~12/18,
   87-100 Toru\'n,
   Poland}
\email{jb@mat.uni.torun.pl}
\email{ahajduk@mat.umk.pl}
\email{skowron@mat.uni.torun.pl}

\address[Karin Erdmann]{Mathematical Institute,
   Oxford University,
   ROQ, Oxford OX2 6GG,
   United Kingdom}
\email{erdmann@maths.ox.ac.uk}

\address[Kunio~Yamagata]{Institute of Engineering, 
    Tokyo University of Agriculture and Technology\\
    Nakacho 2-24-16, Koganei,
    Tokyo 184-8588, Japan}
\email{yamagata@cc.tuat.ac.jp}

\begin{abstract}
We describe the structure and properties of the finite-dimensional
symmetric algebras over an algebraically closed field $K$
which are socle equivalent to the general weighted surface algebras
of triangulated surfaces, 
investigated in 
\cite{ES5}.
n particular, we prove that all these algebras 
are tame periodic algebras of period~$4$.
The main results of this paper form an essential step towards
a classification of all symmetric tame periodic algebras of period 4.

\bigskip




\end{abstract}

\maketitle

\section{Introduction and the main results}%
\label{sec:intro}

Throughout this paper, $K$ will denote a fixed algebraically closed field.
By an algebra we  mean an associative finite-dimensional $K$-algebra
with an identity. 
All algebras here will be symmetric, hence in particular self-injective.
Two self-injective algebras $A$ and $B$ are said to be
\emph{socle equivalent} if the quotient algebras
$A/\!\soc(A)$ and $B/\!\soc(B)$ are isomorphic.

Let $A$ be an algebra.
Given a (right) $A$-module $M$, its \emph{syzygy}
is defined to be 
the kernel $\Omega_A(M)$ of a
minimal
projective cover of $M$.
A module $M$  is periodic if $\Omega_A^n(M)\cong M$ for some $n\geq 1$.
An algebra $A$ is defined to be \emph{periodic}
if it is periodic viewed as a module over
the enveloping algebra $A^e = A^{\op} \otimes_K A$,
or equivalently,  as an $A$-$A$-bimodule.
It is known that if $A$ is a periodic algebra
of period $n$ then for any indecomposable
non-projective module $M$ in $\mod A$
the syzygy $\Omega_A^n(M)$ is isomorphic to $M$.

Finding or possibly classifying periodic algebras
is an important problem, for motivation we refer to the survey article \cite{ES2}
and the introductions of \cite{BES,ES3}).
Periodicity of an algebra, and its 
period, 
are invariant under derived equivalences \cite{Ric}
(see also \cite{ES2}).
Therefore to study periodic algebras we may assume
that the algebras are basic and indecomposable.

We are concerned with the classification of all periodic
tame symmetric algebras.
In \cite{Du1} Dugas proved that every representation-finite
self-injective algebra, without simple blocks, is a periodic algebra.
The representation-infinite, basic, indecomposable,
periodic algebras of polynomial growth were classified
by Bia\l kowski, Erdmann and Skowro\'nski in \cite{BES}
(see also \cite{S1,S2}).
It would be interesting to classify all indecomposable
periodic symmetric tame algebras of non-polynomial growth.
It is conjectured in \cite[Problem]{ES3} 
that every such an algebra has period $4$.

Recently, weighted surface algebras 
are introduced and studied in 
\cite{ES3}--\cite{ES6}, 
a new class of tame symmetric periodic algebras,
of period $4$.
The present paper is a sequel of this work.

We recall the general concept of weighted surface algebras and their socle 
deformations and describe the main results of this paper; for details we refer to Section 2.
Let $S$ be a surface, that is  a compact connected 
$2$-dimensional real manifold, orientable or non-orientable,
with or without boundary, and $T$ a triangulation of $S$.
We say that $(S, \vec{T})$ is a directed triangulated surface if $S$
is a surface, $T$ is a triangulation of $S$ with at least $2$ edges,
and $\vec{T}$ is an arbitrary choice of orientations of the triangles in $T$.
To such $(S, \vec{T})$ one associates \cite{ES3} a triangulation
quiver $Q=(Q(S, \vec{T}),f)$, where $f$ is the permutation of the arrows describing the orientation $\vec{T}$
of the triangles in $T$.
The weighted surface algebra 
$\Lambda(S, \vec{T}, m_{\bullet}, c_{\bullet})$
is a quotient algebra   of $KQ$, here $m_{\bullet}, c_{\bullet}$ describe multiplicities and scalar parameters.
Furthermore,   if the border $\partial S$ of $S$ is not empty, we have  the set 
$\partial(Q)$
of border vertices corresponding to the boundary edges 
of $T$, and a border function
$b_{\bullet} : \partial(Q) \to K$. 
This gives rise to the  socle deformed weighted surface algebra
$\Lambda(S,\vec{T},m_{\bullet},c_{\bullet},b_{\bullet})$.

In the original version \cite{ES3}, we had assumed that $Q$ has at least
three vertices. The general version \cite{ES5}, \cite{ES5-corr}
is also valid when $Q$ has two vertices, and it 
deals with 
algebras whose Gabriel quiver is not 2-regular. 

 Theorem 1.4 of \cite{ES3}, 
describes algebras which are socle equivalent to weighted surface
algebras, in terms of the original definition so that in particular
$Q$ has at least three vertices, proving the following:
Assume  $A$ is a basic indecomposable symmetric algebra
over $K$ which is socle equivalent but not isomorphic to a weighted surface
algebra $\La(S, \vec{T}, m_{\bullet}, c_{\bullet})$. Then the surface
has non-empty boundary, the field has characteristic $2$, and 
$A$ is isomorphic to a socle deformed weighted surface algebra
$\Lambda(S, \vec{T}, m_{\bullet}, c_{\bullet}, b_{\bullet})$.
 The algebra $A$ is
tame of non-polynomial growth, and $A$ is periodic as an algebra,
of period four.

\bigskip

In this paper we will establish the same results for the general 
version as in  \cite{ES5}, \cite{ES5-corr}. 
Here the quiver may also have only 2 vertices, and the Gabriel quiver
of the algebra will in general be a proper subquiver of the quiver
$Q$. 
So let $\Lambda(S, \vec{T}, m_{\bullet}, c_{\bullet})$ be a (general) weighted
surface algebra. We assume throughout that it is not the singular disc
algebra $D(1)$ introduced in \cite{ES5}, Example 3.1 (as this is not periodic, see
\cite{ES5}, Theorem 1.3).

The following three theorems are the main results of the paper.

\begin{main-theorem}
\label{th:main1}
Let $\Lambda = \Lambda(S,\vec{T},m_{\bullet},c_{\bullet},b_{\bullet})$
be a socle deformed weighted surface algebra over an algebraically
closed field $K$,
with non-zero border function $b_{\bullet}$,
and $|Q_0| \geqslant 2$.
Then the following hold:
\begin{enumerate}[\rm (i)]
 \item
  $\Lambda$ is a finite-dimensional symmetric algebra.
 \item
  $\Lambda$ is socle equivalent to 
  $\Lambda(S,\vec{T},m_{\bullet},c_{\bullet})$,
  and isomorphic if $K$ is of characteristic different from $2$.
 \item
  $\Lambda$ is a tame algebra.
 \item
  $\Lambda$ is a periodic algebra of period $4$.
\end{enumerate}
\end{main-theorem}

\begin{main-theorem}
\label{th:main2}
Let $A$ be a basic, indecomposable, symmetric
algebra with Grothendieck group $K_0(A)$ of rank at least $2$ over
an algebraically closed field $K$.
Assume that $A$ is socle equivalent to a weighted surface algebra 
$\Lambda(S,\vec{T},m_{\bullet},c_{\bullet})$.
\begin{enumerate}[\rm (i)]
 \item
  If the border $\partial S$ of $S$ is empty, then $A$ is isomorphic to
  $\Lambda(S,\vec{T},m_{\bullet},c_{\bullet})$.
 \item
  Otherwise, $A$ is isomorphic to
  $\Lambda(S,\vec{T},m_{\bullet},c_{\bullet},b_{\bullet})$
  for some border function $b_{\bullet}$ of $(Q(S,\vec{T}),f)$.
\end{enumerate}
\end{main-theorem}

\begin{main-theorem}
\label{th:main3}
Let $\Lambda = \Lambda(S,\vec{T},m_{\bullet},c_{\bullet},b_{\bullet})$
be a socle deformed weighted surface algebra 
with non-zero border function $b_{\bullet}$
over an algebraically closed field $K$ of characteristic $2$.
Then
\begin{enumerate}[\rm (i)]
 \item
  If $|Q_0|\geq 3$ then $\Lambda$ is not isomorphic to a
	weighted surface algebra. 
 \item
  Suppose $|Q_0|=2$, then $\Lambda$ is isomorphic to a weighted
		surface algebra $\Lambda(S, \vec{T}, m_{\bullet}, c_{\bullet})$
		if and only if $\Lambda$ is isomorphic to an algebra
  $Q(2\cB)_3^t(b)$ with $t$ even and $b \in K^*$.
\end{enumerate}
\end{main-theorem}

Here, $Q(2\cB)_3^t(b)$ is an algebra of quaternion type
with two simple modules from \cite{E}
(see Section~\ref{sec:9} for details).

\smallskip

The paper is organized as follows.
In Section~\ref{sec:2} we introduce the socle deformed weighted surface algebras
and describe their basic properties.
In Section~\ref{sec:3} we prove that these are periodic algebras of period $4$,
completing the proof of Theorem~\ref{th:main1}.
Section~\ref{sec:4} is devoted to the proof of Theorem~\ref{th:main2}.
In Sections \ref{sec:5}, \ref{sec:6}, \ref{sec:7} we provide technical reductions
results for the proof of Theorem~\ref{th:main3}.
In Sections \ref{sec:8} and \ref{sec:9} we classify  the socle equivalences
for two special classes of algebras of quaternion type with two simple modules.
In the final Section~\ref{sec:10} we complete the proof of Theorem~\ref{th:main3}.
For general background on the relevant representation theory
we refer to the books \cite{ASS,E,SS,SY}.

\section{Socle deformed weighted surface algebras}
\label{sec:2}

Recall that a quiver is a quadruple $Q = (Q_0,Q_1,s,t)$ where
$Q_0$ is a finite set of vertices, $Q_1$ is a finite set of arrows,
and where $s,t$ are maps $Q_1 \to Q_0$ associating to each arrow
$\alpha \in Q_1$ its source $s(\alpha)$ and target $t(\alpha)$.
We say that $\alpha$ starts at $s(\alpha)$ and ends at $t(\alpha)$.
We assume throughout that any quiver is connected.

Let $Q$ be a quiver.
We denote by $K Q$ the path algebra of $Q$ over $K$.
The underlying space has 
basis 
consisting of
the set of all paths (of length $\geqslant 0$) in $Q$.
Let $R_Q$ be the ideal of $K Q$
generated by all paths in $Q$ of length $\geqslant 1$.
For each vertex $i \in Q_0$,
let $e_i$ be the path in $Q$ of length $0$ at $i$,
then the $e_i$ are pairwise orthogonal primitive idempotents,
and their sum is the identity of $K Q$.
We will consider algebras of the form $A = K Q/I$
where $I$ is an ideal in $K Q$ which contains $R_Q^m$ for some $m \geqslant 2$,
so the algebra $A$ is finite-dimensional and basic.
The Gabriel quiver $Q_A$ of $A$ is then the full subquiver of $Q$
obtained from $Q$ by removing all arrows $\alpha$ with
$\alpha + I \in R_Q^2 + I$.

A quiver $Q$ is $2$-regular if for each vertex $i \in Q_0$ there
are precisely two arrows starting at $i$ and two arrows
ending at $i$.
Such a quiver has an involution $(-) : Q_1 \to Q_1$
such that for each arrow $\alpha \in Q_1$, the arrow $\bar{\alpha}$
is the arrow $\neq\alpha$ with $s(\bar{\alpha}) = s(\alpha)$.

A triangulation quiver is a pair $(Q,f)$ where $Q$ is a (finite)
connected $2$-regular quiver, with at least two vertices,
and where $f$ is a fixed permutation on the set of arrows $Q_1$
such that $t(\alpha) = s(f(\alpha))$ for any arrow $\alpha \in Q_1$,
and $f^3$ is the identity.
The permutation $f$ uniquely determines a permutation $g$
on $Q_1$, defined by $g(\alpha) = \overbar{f(\alpha)}$ 
for any arrow $\alpha \in Q_1$.

It was proved in \cite{ES3} (see also \cite{ES8}) that the triangulation
quivers are precisely the  quivers  $(Q(S, \vec{T}),f)$
constructed from a triangulation of a compact real surface $S$,
with or without boundary, and where $\vec{T}$ is an arbitrary
choice of orientation of triangles in $T$.

Let $(Q,f)$ be a triangulation quiver.
For each arrow $\alpha \in Q_1$,
we denote by $\cO(\alpha)$ the $g$-orbit of $\alpha$ in $Q_1$,
and set $n_{\alpha} = n_{\cO(\alpha)} = |\cO(\alpha)|$.
We denote by $\cO(g)$ the set of all $g$-orbits in $Q_1$.
Following \cite{ES3}, a function
\[ 
  m_{\bullet} : \cO(g) \to \bN^* = \bN \setminus \{ 0 \}
\]
is called a \emph{weight function} of $(Q,f)$, and a function
\[ 
  c_{\bullet} : \cO(g) \to K^* = K \setminus \{ 0 \}
\]
is called a \emph{parameter function} of $(Q,f)$.
We write briefly $m_{\alpha} = m_{\cO(\alpha)}$ and
$c_{\alpha} = c_{\cO(\alpha)}$ for $\alpha \in Q_1$.
We say that an arrow $\alpha$ of $Q$ is \emph{virtual}
if $m_{\alpha} n_{\alpha} = 2$
(see \cite[Definition~2.6]{ES5}).
Following \cite{ES5}, we assume that any weight functions
$m_{\bullet}$ of $(Q,f)$ satisfies the following conditions:
\begin{enumerate}[\rm \ \ (1)]
 \item
  $m_{\alpha} n_{\alpha} \geqslant 2$ for all arrows $\alpha \in Q_1$,
 \item
  $m_{\alpha} n_{\alpha} \geqslant 3$ for all arrows $\alpha \in Q_1$
  such that $\bar{\alpha}$ is virtual and $\bar{\alpha}$ is not a loop,
 \item
  $m_{\alpha} n_{\alpha} \geqslant 4$ for all arrows $\alpha \in Q_1$
  such that $\bar{\alpha}$ is virtual and $\bar{\alpha}$ is a loop.
\end{enumerate}
Condition (1) is a general assumption, and (2) and (3) are 
needed to eliminate some small algebras (see \cite[Section~3]{ES5}).
For each arrow $\alpha \in Q_1$, we fix
\begin{align*} 
 A_{\alpha}:=  \alpha g(\alpha)\ldots g^{m_{\alpha}n_{\alpha}-2}(\alpha), &
   \mbox{ the path  along the $g$-cycle of $\alpha$ 
                     of length $m_{\alpha}n_{\alpha}-1$,}
\\
B_{\alpha}:=  \alpha g(\alpha)\ldots g^{m_{\alpha}n_{\alpha}-1}(\alpha), &
   \mbox{ the path  along the $g$-cycle of $\alpha$ 
                     of length $m_{\alpha}n_{\alpha}$.} 
\end{align*}

A loop $\alpha$ in $Q_1$ with $f(\alpha) = \alpha$ is said to be
a \emph{border loop} and the vertex $s(\alpha)$ a border vertex,
of $(Q,f)$.
We denote by $\partial(Q,f)$ the set of all border vertices of $(Q,f)$,
and call it the \emph{border} of $(Q,f)$.
Assume $\partial(Q,f)$ is not empty. A function
\[ 
  b_{\bullet} : \partial(Q,f) \to K 
\]
is said to be a \emph{border function} of $(Q,f)$.

The following is the  corrected version of Definition 2.8 of 
\cite{ES5}. The details for the correction are given in \cite{ES5-corr}.

\begin{definition}
\label{def:socdef}
Let $(Q,f)$ be a triangulation quiver with $\partial(Q,f)$ not empty,
and
$m_{\bullet}$, $c_{\bullet}$, $b_{\bullet}$
be weight, parameter, border functions
of $(Q,f)$.
Then the 
\emph{socle deformed weighted triangulation algebra}
$\Lambda(Q,f,m_{\bullet},c_{\bullet},b_{\bullet}) = K Q/I$
is defined by the ideal
$I = I(Q,f,m_{\bullet},c_{\bullet},b_{\bullet})$
of the path algebra $K Q$ generated by:
\begin{enumerate}[\rm \ \ (1)]
 \item
	$\alpha f(\alpha) - c_{\bar{\alpha}}A_{\bar{\alpha}}$ 
	for all arrows $\alpha$ of $Q$,
	which are not border loops,
 \item
	$\alpha^2 - c_{\bar{\alpha}}A_{\bar{\alpha}} - b_{s(\alpha)}B_{\bar{\alpha}}$ 
	for all border loops $\alpha$ of $Q$,
 \item
	$\alpha f(\alpha) g(f(\alpha))$ \ for all arrows $\alpha$ of $Q$, except if 
	$f^2(\alpha)$ is virtual, or $f(\ba)$ is virtual and $m_{\ba}=1, n_{\ba}=3$,
 \item
	$\alpha g(\alpha)f(g(\alpha))$ for all arrows $\alpha$ of $Q$
	except if $f(\alpha)$ is  virtual, or $f^2(\alpha)$ is virtual and $m_{f(\alpha)}=1, n_{f(\alpha)}=3$.
\end{enumerate}
If $b_{\bullet}$ is a zero border function 
($b_{i} = 0$ for all $i \in\partial(Q,f)$ ),
then
$\Lambda(Q,f,m_{\bullet},c_{\bullet},b_{\bullet}) 
 = \Lambda(Q,f,m_{\bullet},c_{\bullet})$
is a weighted triangulation algebra as defined in \cite[Definition~2.8]{ES5}.
Moreover, if $(Q,f) = (Q(S,\vec{T}),f)$
for a directed triangulated surface
$(S,\vec{T})$ with non-empty boundary $\partial S$, then
$\Lambda(S,\vec{T},m_{\bullet},c_{\bullet},b_{\bullet})
 = \Lambda(Q(S,\vec{T}),f,m_{\bullet},c_{\bullet},b_{\bullet})$
is said to be a 
\emph{socle deformed weighted surface algebra}.
\end{definition}

As mentioned, any triangulation quiver $(Q, f)$ comes from a directed triangulated surface. We will use  the name
'weighted surface algebra' in general, but most of the time work with $(Q, f)$.

\begin{example}
\label{ex:2.2}
Let $T$ be the triangulation
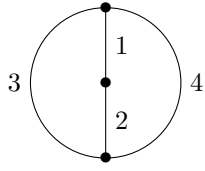
\[
\begin{tikzpicture}[scale=1,auto]
\coordinate (o) at (0,0);
\coordinate (c) at (0,-1);
\coordinate (a) at (0,1);
\coordinate (b) at (1,0);
\coordinate (d) at (-1,0);
\draw (a) to node {$1$} (o);
\draw (o) to node {$2$} (c);
\draw (b) arc (0:360:1) node [right] {$4$};
\draw (d) node [left] {$3$};
\node (d) at (0,-1) {$\bullet$};
\node (c) at (0,0) {$\bullet$};
\node (b) at (0,1) {$\bullet$};
\end{tikzpicture}
\]
of the unit disc $D = D^2$ in $\bR^2$
by two triangles
and $\vec{T}$ the orientation
$(1\ 2\ 3)$, 
$(1\ 4\ 2)$
of triangles in $T$.
Then the  triangulation quiver
$(Q,f) = (Q(D,\vec{T}),f)$ is the quiver
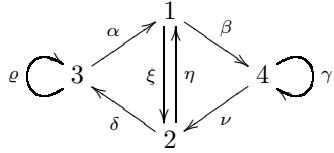
\[
  \xymatrix@R=1pc{
   & 1 \ar[rd]^{\beta} \ar@<-.5ex>[dd]_{\xi}  \\
   3   \ar@(ld,ul)^{\varrho}[]  \ar[ru]^{\alpha} 
   && 4   \ar@(ru,dr)^{\gamma}[]  \ar[ld]^{\nu}   \\
   & 2 \ar[lu]^{\delta}  \ar@<-.5ex>[uu]_{\eta}
  } 
\]
with $f$-orbits
$(\alpha\ \xi\ \delta)$,
$(\beta\ \nu\ \eta)$,
$(\varrho)$,
$(\gamma)$.
Then $g$ has two 
orbits, 
$\cO(\alpha) = (\alpha\ \beta\ \gamma\ \nu\ \delta\ \varrho)$ 
and
$\cO(\xi) = (\xi\ \eta)$.
Observe that $\varrho$ and $\gamma$ are the border loops
of $(Q,f)$, corresponding to the border edges of $D$, and hence
$\partial(Q,f) = \{3,4\}$.
Let
$m_{\bullet} : \cO(g) \to \bN^*$
be the weight function
with 
$m_{\cO(\alpha)}  = 3$
and
$m_{\cO(\xi)}  = 1$,
$c_{\bullet} : \cO(g) \to K^*$
the parameter function
with 
$c_{\cO(\alpha)}  = c$
and
$c_{\cO(\xi)}  = 1$,
and
$b_{\bullet} : \partial(Q,f) \to K$
the border function 
with 
$b_3=1$ and $b_4=0$.
Observe that $\xi$ and $\eta$ are virtual arrows.
Then  the associated algebra
$\Lambda=\Lambda(Q,f,m_{\bullet},c_{\bullet},b_{\bullet})$
is given by the quiver $Q$ and the relations:
\begin{align*}
\alpha \xi 
&= c (\varrho \alpha \beta \gamma \nu \delta)^2 \varrho \alpha \beta \gamma \nu ,
\!\!\!\!  
&
\nu \eta 
&= c (\gamma \nu \delta \varrho \alpha \beta)^2 \gamma \nu \delta \varrho \alpha ,
\!\! 
&
\xi \delta
&= c (\beta \gamma \nu \delta \varrho \alpha)^2 \beta \gamma \nu \delta \varrho ,
\!\! 
&
\eta \beta
&= c (\delta \varrho \alpha \beta \gamma \nu)^2 \delta \varrho \alpha \beta \gamma ,
\\
\delta \alpha &=\eta,
\qquad \beta \nu =\xi,
&
\varrho^2 
&= c (\alpha \beta \gamma \nu \delta \varrho)^2 \alpha \beta \gamma \nu \delta
+(\alpha \beta \gamma \nu \delta \varrho)^3 ,
\!\!\!\!\!\!\!\!\!\!\!\!\!\!\!\!\!\!\!\!\!\!\!\!\!\!\!\!\!\!\!\!\!\!\!\!&&
&
\gamma^2
&= c (\nu \delta \varrho \alpha \beta \gamma)^2 \nu \delta \varrho \alpha \beta ,
\\
\alpha  \xi \eta &= 0 ,
\qquad \xi \delta \varrho = 0 , 
&
\nu \eta \xi &= 0 ,
\qquad \eta \beta \gamma = 0 ,
&
\varrho^2 \alpha &= 0 ,
&
\!\!   \gamma^2 \nu &= 0 ,
\\
\xi \eta \beta &= 0 ,
\qquad \varrho \alpha \xi = 0 ,
&
\eta \xi \delta &= 0 ,
\qquad \gamma \nu \eta = 0 ,
&
\delta \varrho^2 &= 0 ,
&
\!\!  \beta \gamma^2 &= 0 .
\
\end{align*}
We note that the Gabriel quiver $Q_{\Lambda}$ of $\Lambda$
is of the form
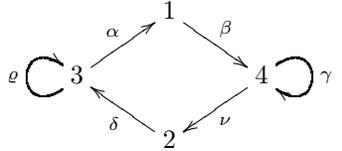
\[ 
 \vcenter{
  \xymatrix@R=1pc{
   & 1 \ar[rd]^{\beta} \\
   3   \ar@(ld,ul)^{\varrho}[]  \ar[ru]^{\alpha} 
   && 4   \ar@(ru,dr)^{\gamma}[]  \ar[ld]^{\nu}   \\
   & 2 \ar[lu]^{\delta} 
  } 
 }
 .
\]
Then $\Lambda$ is also given by $Q_{\Lambda}$ and the relations
obtained from the above relations by replacing $\eta = \delta \alpha$
and $\xi = \beta \nu$.
It follows from the proposition below that $\Lambda$ is a symmetric
algebra of dimension 
$\dim_K \Lambda = m_{\cO(\alpha)} n_{\cO(\alpha)}^2
  + m_{\cO(\xi)} n_{\cO(\xi)}^2 = 3 \cdot 6^2 + 1 \cdot 2^2 = 112$.
\end{example}

\begin{proposition}
\label{prop:2.3}
Let $\Lambda = \Lambda(Q,f,m_{\bullet},c_{\bullet},b_{\bullet})$
be a socle deformed weighted surface algebra.
Then the following hold:
\begin{enumerate}[\rm (i)]
 \item
  $\Lambda$ is a finite-dimensional algebra
  with $\dim_K \Lambda = \sum_{\cO \in \cO(g)} m_{\cO} n_{\cO}^2$.
 \item
  $\Lambda$ is socle equivalent to 
  $\Lambda(Q,f,m_{\bullet},c_{\bullet},0)$.
 \item
  $\Lambda$ is a symmetric algebra.
 \item
  $\Lambda$ is a tame algebra.
 \item
  If $K$ is of characteristic different from $2$, then
  $\Lambda$ is isomorphic to $\Lambda(Q,f,m_{\bullet},c_{\bullet},0)$.
\end{enumerate}
\end{proposition}

\begin{proof} Parts (i), (ii) and (iii) are the same as parts (i), (ii) and (iv) of Proposition 8.1 of \cite{ES3}, and also Proposition 4.13 of \cite{ES5}.
By part (ii), to prove tameness, it suffices to deal with the case when $b_{\bullet}$ is zero. In this case, tameness
is proved in \cite{ES5} (which is not affected by the correction in \cite{ES5-corr}). 
The proof of part (v) is the same as Proposition 8.2 of \cite{ES3}. In \cite{ES3} it is assumed that $|Q_0|\geq 3$ but same proofs work for the quiver with two vertices.
\end{proof}


\section{Periodicity of socle deformed weighted surface algebras}
\label{sec:3}

Throughout this section, $\La$ is a socle deformed weighted surface algebra, $\La = \La(Q, f, m_{\bullet}, c_{\bullet}, b_{\bullet})$. We assume that
the boundary $\partial(Q, f)$ is non-empty. 
Furthermore, we may assume that when $\beta$ is a virtual arrow then $c_{\beta}=1$ (see \cite{ES5}, the comments following Definition 2.8). 

The general construction of
the first three terms of a minimal bimodule resolution for
weighted surface algebras  is explained in detail in \cite{ES3} and again in \cite{ES5} and we will not repeat it.
In order to identify the  terms, by Happel's result \cite{Ha}, we must determine dimensions of
the spaces $\Ext^n_{\La}(S_i, S_j)$ for all simple modules $S_i, S_j$ of $\La$. 
We will show that every simple module has  $\Omega$-period four,
most of  this can be taken  from previous 
proofs.

\begin{proposition}
\label{prop:3.1}
Let $i$ be a  vertex of $Q$ where the arrows  $\alpha$, $\bar{\alpha}$ starting at $i$ are both not virtual.
Then there is an exact sequence in $\mod \Lambda$
\[
  0 \rightarrow
  S_i \rightarrow
  P_i \xrightarrow{\pi_3}
  P_{t(f(\alpha))} \oplus P_{t(f(\bar{\alpha}))} \xrightarrow{\pi_2}
  P_{t(\alpha)} \oplus P_{t(\bar{\alpha})} \xrightarrow{\pi_1}
  P_i \rightarrow
  S_i \rightarrow
  0,
\]
which give rise to a minimal projective resolution of $S_i$ in $\mod \Lambda$.
In particular, $S_i$ is  periodic of period $4$.
\end{proposition}

\begin{proof} If $i$ is not a border vertex then this is the  same as the proof of Proposition 5.4 of \cite{ES5}, which  refers to 7.1 \cite{ES3}. 
Suppose 
 $\alpha$ (say) is a border loop.  When $|Q_0|\geq 3$ (which is assumed in 
 \cite{ES3}, see Section 4) then the proof of Proposition 9.1 of \cite{ES3} can be used verbatim.
This leaves the case  when $Q$ has two vertices, and then $Q$ is of the form
\[
  \xymatrix{
 1
	\ar@<.5ex>[r]^{\beta}
    \ar@(dl,ul)^{\alpha}[]
& 2
    \ar@<.5ex>[l]^{\gamma}
\ar@(dr, ur)_{\sigma}  }
 .
\]

Let $\alpha$ be the  border loop. That is, $f=(\alpha)(\beta \ \sigma \ \gamma)$ and $g = (\alpha \ \beta \ \gamma)(\sigma)$. 
We have $m_{\sigma}\geq 3$ and $m_{\alpha}\geq 1$, or if $m_{\sigma}=2$ then we must have $m_{\alpha}\geq 2$  (by the assumption, see Example 3.1 in 
\cite{ES5}). 

 Let $m=m_{\alpha}$ and $c=c_{\alpha}$, $b=b_i$. The relations at vertex $1$ are
$$\alpha^2 = c A_{\beta} + bB_{\beta}, \ \ \beta\sigma = c A_{\alpha}.
$$
We take $\Omega(S_1) = \alpha\La + \beta\La$, and we take $\pi_1(x, y)= \alpha x + \beta y$, and we identify the kernel of $\pi_1$ with
$\Omega^2(S_1)$. This has dimension
equal to $\dim P_2 + 1 = 3m_{\alpha} + m_{\sigma} + 1$.
The above relations give two elements in the kernel of $\pi_1$, namely
$$\vf:= (\alpha, -cA_{\beta}' - bA_{\gamma}), \ \ 
\psi:=  (-cA_{\alpha}', \sigma).$$
Here $A_{\beta}'$ is the monomial with $\beta A_{\beta}'=A_{\beta}$, similarly we define $A_{\alpha}'$. 
Hence $\vf\La  + \psi\La \subseteq \Omega^2(S_1)$, and we will show that both modules
have the same dimension. 

(1)\, By projecting onto the first coordinate we get an exact sequence 
$$
 0\to \Ker (p_1) \to \vf\La \stackrel{p_1}\longrightarrow \alpha\La \to 0 .
$$
The kernel of $p_1$ is given by 
$\{ (0, (-cA_{\beta}' - bA_{\gamma})x) \ \mid \alpha x=0\}$.
By the lemma below,  we have $x=\alpha\beta x'$ for $x'\in \La$, and we find that $\Ker (p_1)$ is 1-dimensional, spanned by $(0, B_{\gamma})$. Hence
$\dim \vf\La = 3m_{\alpha}+2.$

(2)\, We have similarly an exact sequence, given by projecting onto the second coordinate,
$$0\to \Ker (p_2) \to \psi\La \stackrel{p_2}\longrightarrow \sigma\La\to 0$$
and we find, using the second part of the  lemma below, that the kernel of $p_2$ is 1-dimensional,
spanned by $(B_{\beta}, 0)$. This shows that
$\dim \psi\La = m_{\sigma} +2.$

(3)\, We analyse the intersection of $\vf\La$ and $\psi\La$. First we observe that it contains the socle of $P_1\oplus P_2$.  We have that
$$\psi\gamma = (-cA_{\beta}, \sigma\gamma), \ \ \vf\alpha = (\alpha^2,  -cA_{\gamma})$$
and it follows, using the relations, that
$\psi\gamma + \vf\alpha = (bB_{\beta}, 0)$
which is in $\vf\La \cap \psi\La$. 
Using the two exact sequences, one checks that the submodule generated by $\vf\alpha$ is all of the intersection. It is 3-dimensional, and then from (1) and (2) it follows that
$\dim (\vf\La + \psi\La) = \dim\Omega^2(S_1)$ as required.

The case   $m_{\alpha}=1$ and $m_{\sigma} = 3$ is special,  here 
we  may take $c_{\sigma}=\lambda$ and $c_{\beta}=1$ (see   
Example 3.1 of \cite{ES5}), and since we exclude the singular disc algebra we have $\lambda\neq 1$.

Then $\psi\sigma = (-\beta\sigma, \sigma^2) = (-\alpha\beta,  \lambda^{-1}\gamma\beta)$ and 
$-\vf\beta = (-\alpha\beta,  \gamma\beta - bB_{\gamma})$. 
Since $\lambda\neq 1$ these are independent modulo $\vf\alpha\La$ and one gets $\vf\La + \psi\La$ has dimension $7 = \dim \Omega^2(S_1)$.

We define a surjective homomorphism $\pi_2: P_1\oplus P_2\to  \vf\La + \psi\La$ by $\pi_2(x, y):= \vf x + \psi y$. 
We have seen above that
$\vf\alpha + \psi\gamma - (bB_{\gamma}, 0) = 0$. 
Moreover, $(bB_{\alpha}, 0) =  b\lambda^{-1}\vf\alpha^2$ and therefore we have
$\Theta:= (\alpha - b\lambda^{-1}\alpha^2, \gamma) \in  \Ker (\pi_2)$. 

Now, $\alpha':= \alpha -b\lambda^{-1}\alpha^2$ and $\gamma$ are independent elements not in the radical of $\La$, representing all arrows ending at vertex $1$. 
Therefore
$\Theta\La \cong \Omega^{-1}(S_1)$. This has dimension equal to $\dim \Ker(\pi_2)$, and equality follows.
This completes the proof.
\end{proof}

In the proof we have used:

\smallskip

\begin{lemma} \label{lem:3.2}We have $\Omega(\alpha\La) = \alpha\beta\La$ and $\Omega(\sigma\La) = \gamma\alpha\La$.
\end{lemma}

\begin{proof}  
We identify $\Omega(\alpha\La)$ with $\{ x\in e_1\La \mid \alpha x=0\}$. 
By the zero relations, $\alpha^2\beta = 0$ and hence $\alpha\beta\La \subseteq \Omega(\alpha\La)$. 
We get equality, by observing that both spaces have the same dimension (namely $\dim e_1\La - (3m_{\alpha} +1) = 3m_{\beta}-1 = \dim (\alpha\beta\La)$).
Similarly one proves the second statement.
\end{proof}


\bigskip

Next we consider a simple module $S_i$ where one of the arrows at $i$ is virtual.

\medskip

\begin{proposition}\label{prop:3.3}
Let $i$ be a  vertex of $Q$ with arrows $\alpha, \ba$ starting at $i$, and where the arrow  $\bar{\alpha}$ is a virtual loop.
Then there is an exact sequence in $\mod \Lambda$
\[
  0 \rightarrow
  \Omega^{-1}(S_i) \rightarrow
  P_{t(\alpha)}
  \rightarrow
  P_{t(\alpha)}
 \rightarrow
  \Omega(S_i) \rightarrow
  0,
\]
which gives rise to a minimal projective resolution of $S_i$ in $\mod \Lambda$.
In particular, $S_i$ is  periodic of period $4$.
\end{proposition}

\begin{proof} This is the same as the proof of Lemma 5.5 in \cite{ES5}, taking  the correction in 
\cite[Lemma 5.1]{ES5-corr}. 
\end{proof}

\bigskip

\begin{proposition}\label{prop:3.4}
Let $i$ be a  vertex of $Q$ with arrows $\alpha, \ba$ starting at $i$ and where  $\bar{\alpha}$  is virtual but not a loop. 
Then there is an exact sequence in $\mod \Lambda$
\[
  0 \rightarrow
  \Omega^{-1}(S_i) \rightarrow
  P_{t(f(\ba))}  \rightarrow
  P_{t(\alpha)} 
   \rightarrow
  \Omega(S_i) \rightarrow
  0,
\]
which gives rise to a minimal projective resolution of $S_i$ in $\mod \Lambda$.
In particular, $S_i$ is  periodic of period $4$.
\end{proposition}

\begin{proof}
This is the same as that of Lemma 5.6 of \cite{ES5}, together with the correction in 
\cite[Lemma 5.2]{ES5-corr}.
\end{proof}

\bigskip


\medskip

We construct now the first steps of a minimal
projective bimodule resolution of $\Lambda$.  As explained in in \cite{ES5}, Section 5, 
they are as follows. Note that we must use the Gabriel quiver $Q_{\La}$, recall this is the quiver obtained from $Q$ by removing the virtual arrows.
\[
  \bP_3 \xrightarrow{S}
  \bP_2 \xrightarrow{R}
  \bP_1 \xrightarrow{d}
  \bP_0 \xrightarrow{d_0}
  \Lambda \to 0
\]
where
\begin{align*}
  \bP_0
     & = \bigoplus_{i \in Q_0} \Lambda e_i \otimes e_i \Lambda ,
  \\
  \bP_1
     &= \bigoplus_{\alpha \in (Q_{\La})_1} \Lambda e_{s(\alpha)} \otimes e_{t(\alpha)} \Lambda , \\
     \bP_2 
     &= \bigoplus_{\alpha \in (Q_{\La})_1}  \Lambda e_{s(\ba)} \otimes e_{t(f(\ba))} \Lambda ,\\
       \bP_3
         &  = \bigoplus_{i \in Q_0} \Lambda e_i \otimes e_i \Lambda.
\end{align*}
The homomorphism $d_0$
is defined by
$d_0 ( e_i \otimes e_i ) = e_i$ for all $i \in Q_0$,
and the homomorphism $d : \bP_1 \to \bP_0$
is defined by
\[
  d \big( e_{s(\alpha)} \otimes e_{t(\alpha)} \big)
    = \alpha \otimes e_{t(\alpha)} - e_{s(\alpha)} \otimes \alpha
\]
for any arrow $\alpha$ in $(Q_{\La})_1$.
In particular, we have
$\Omega_{\Lambda}^1(\Lambda) = \Ker d_0$
and
$\Omega_{\Lambda}^2(\Lambda) = \Ker d$.

For each arrow $\alpha$ in $(Q_{\La})_1$, we define the element
$\mu_{\alpha} = e_{s(\ba)} \mu_{\alpha} e_{t(f(\ba))}$
as follows
\begin{align*}
  \mu_{\alpha} &= \ba f(\ba) - c_{\alpha} A_{\alpha}
  &&
            \mbox{if $\ba$ is not a border loop},
  \!\!\!\!\!\!\!\!\!\!\!
  \\
  \mu_{\alpha} &= \ba^2 - c_{\alpha} A_{\alpha} - b_i B_{\alpha}
  \!\!\!\!\!\!\!\!\!\!\!\!\!\!\!\!\!\!\!\!\!\!\!\!
  &&
            \mbox{if $\ba$ is a border loop}.
\end{align*}

Since we work with the Gabriel quiver $Q_{\La}$, we must make substitutions. 
First, since $\alpha$ is not virtual, no arrow in the $g$-cycle of $\alpha$ is virtual, and therefore $A_{\alpha}$ is a path in 
$Q_{\La}$. If $\ba$ is virtual then we substitute $\ba = \alpha f(\alpha)$ (using that $c_{\ba}=1$).  
Note that if $\ba$ is virtual then $f(\alpha)$ is not virtual (see Remark 4.2 of \cite{ES5}). Similarly we substitute $f(\alpha) = g(\alpha)f(g(\alpha))$ in case
$f(\alpha)$ is virtual. Recall that $\alpha, f(\alpha)$ cannot be both virtual. 

\medskip

We define the homomorphism $R : \bP_2 \to \bP_1$
in $\mod \Lambda^e$ by
\[
  R\big( e_{s(\ba)} \otimes e_{t(f(\ba))}\big) = \varrho (\mu_{\alpha})
\]
for any arrow $\alpha$ in $(Q_{\La})_1$,
where $\varrho : K Q_{\La} \to \bP_1$
is the $K$-linear homomorphism which is defined on paths by
$$\alpha_1\ldots \alpha_r \mapsto \sum_{k=1}^r \alpha_1\ldots \alpha_{k-1}\otimes \alpha_{k+1}\ldots \alpha_r$$ 
(see Section 3 of \cite{ES3}).  Then 
$\Im R \subseteq \Ker d$ (see Lemma 3.4 of \cite{ES3}).

\begin{lemma}
\label{lem:3.5}
The homomorphism $R : \bP_2 \to \bP_1$
induces a projective cover
$\Omega_{\Lambda^e}^2(\Lambda)$
in $\mod \Lambda^e$.
In particular, we have $\Omega_{\Lambda^e}^3(\Lambda) = \Ker R$.
\end{lemma}

\begin{proof}
This follows by the arguments  in the proof of
Lemma 7.2 and 9.2 of \cite{ES3}.
\end{proof}

\medskip

For each vertex $i \in Q_0$, for $\alpha, \ba$ with $s(\alpha)=i=s(\ba)$, we consider the element in $\bP_2$
\[
  \psi_i =
     \big(e_i \otimes e_{t(f(\alpha))}\big) f^2(\alpha)
     + \big(e_i \otimes e_{t(f(\bar{\alpha}))}\big) f^2(\bar{\alpha})
     - \alpha \big(e_{t(\alpha)} \otimes e_i\big)
     - \bar{\alpha} \big(e_{t(\bar{\alpha})} \otimes e_i\big)
     .
\]
As in \cite{ES5}, Section 5, if $\ba$ is virtual we take the same formula but omit the terms which have virtual arrows, that is
$$
 \psi_i = (e_i\otimes e_{t(f(\ba))})f^2(\ba) - \alpha(e_{t(\alpha)}\otimes e_i) .
$$
If $\alpha$ is virtual then we define similarly
$$\psi_i = (e_i\otimes e_{t(f(\alpha))})f^2(\alpha) - \ba(e_{t(\ba)}\otimes e_i).
$$
Moreover, for each vertex  $i \in \partial(Q,f)$
and the border loop $\alpha$ at $i$,
we consider the elements in $\bP_2$
\begin{align*}
 \psi_i^{(1)} &= (b_i c_{\alpha}^{-1}) (\alpha \otimes \alpha + e_i \otimes \alpha^2),
\\
  \psi_i^{(2)} &= (b_i c_{\alpha}^{-1})^2 (\alpha \otimes \alpha^2 + e_i \otimes \alpha^3),
\\
  \psi_i^{(3)} &= (b_i c_{\alpha}^{-1})^3 (\alpha \otimes \alpha^3).
\end{align*}
Then, for each vertex  $i \in Q$, we define the element
$\bar{\psi}_i$ in $\bP_2$ as follows
\begin{align*}
  \bar{\psi}_i &= \psi_i
  &&
            \mbox{if $i \notin \partial(Q,f)$},
  \\
  \bar{\psi}_i &= \psi_i + \psi_i^{(1)} + \psi_i^{(2)} + \psi_i^{(3)}
  \!\!\!\!\!\!\!\!\!\!\!\!\!\!\!\!\!\!\!\!\!\!\!\!
  &&
            \mbox{if $i \in \partial(Q,f)$}.
\end{align*}
Note that if $\alpha$ is a border loop then $\alpha, \ba$ are not virtual.

We define the homomorphism
$S : \bP_3 \to \bP_2$
in $\mod \Lambda^e$ by
\[
  S( e_i \otimes e_i ) = \bar{\psi}_i
\]
for any vertex $i \in Q_0$.
Then we have the following analogue of
Lemma 7.3 of \cite{ES3}.

\begin{proposition}
\label{prop:3.6}
The homomorphism $S : \bP_3 \to \bP_2$
induces a projective cover of
$\Omega_{\Lambda^e}^3(\Lambda)$
in $\mod \Lambda^e$.
In particular, we have
$\Omega_{\Lambda^e}^4(\Lambda) = \Ker S$.
\end{proposition}

\begin{proof}
This is the same as the proof of Proposition 9.3 in \cite{ES3}.
\end{proof}

\bigskip

The first part of the following lemma holds more generally for self-injective algebras.

\medskip

\begin{lemma}\label{lem:3.7}
{\rm (i)} \ An injective hull of $\La$ as a bimodule is given by $\theta: \La\to  \oplus \La(e_i\otimes e_i)\La$, taking
$$1\mapsto \xi := \sum_{b\in \cB}\sum_{i\in Q_0} be_i\otimes e_ib$$
where $\cB$ is a vector space basis for $\La$ of the form $\cB =\bigcup_{i, j} e_i\cB e_j$.

{\rm (ii)}\, The image of $\theta$ is equal to the kernel of  
$S: \La^e \mapsto   \bP:= \oplus_{\gamma\in (Q_{\La})_1} \langle e_{t(\gamma)}\otimes e_{s(\gamma)}\rangle$, where
$$S(e_i\otimes e_i) = \sum_{t(\gamma)=i} (e_{t(\gamma)}\otimes e_{s(\gamma)})\gamma - \sum_{s(\delta)=i}\delta(e_{t(\delta)}\otimes e_{s(\delta)}).$$
\end{lemma}

\bigskip

\begin{proof} Part (i) is proved in \cite{ES1} [pages 119-120]. 
As well, it is shown there that $S(\xi)=0$. 
To show that the kernel of $S$ is equal to $\theta(\La)$, one uses the arguments in the 
proofs of Lemmas 7.2 and 9.2 of \cite{ES3}, that is, it is the same as
in Lemma \ref{lem:3.5}.
\end{proof}


\begin{theorem}
\label{th:3.8}
There is an isomorphism
$\Omega_{{\Lambda}^e}^4(\Lambda) \cong \Lambda$
in $\mod \Lambda^e$.
In particular, $\Lambda$ is a periodic algebra
of period $4$.
\end{theorem}

\begin{proof}
We proceed as in the proof of
Theorem 7.4 of \cite{ES3},
and use
\cite[part (3) on  pages 119 and 120]{ES1}.
In particular, we fix some
basis
$\cB$ such that $\cB = \bigcup_{i, j \in Q_0} e_i\cB e_j$ of $\Lambda$, and let $\cB_i=e_i\cB$. 
For each $i$ we write $\omega_i$ for the socle element in $\cB_i$ of $e_i \Lambda$,
and consider
the symmetrizing form 
$(-,-) : {\Lambda} \times {\Lambda} \to K$ such that,
for any two elements $x \in \cB_i$ and $y \in \cB$,
we have
\[
  (x,y) = \mbox{\,the coefficient of $\omega_i$ in $x y$},
\]
when $x y$ is expressed
as a linear combination of the elements of $\cB_i$ over $K$.
Moreover, we consider the dual basis
$\cB^*$ of $\cB$ with respect to $(-,-)$.
Then, for each vertex $i \in Q_0$, we define the element of $\bP_3$
\[
  \xi_i = \sum_{b \in \cB_i} b \otimes b^* .
\]
Then we conclude as in the proof of
Theorem 7.4 of \cite{ES3}
that there is a monomorphism
in $\mod \Lambda^e$
\[
  \theta : {\Lambda} \to \bP_3
\]
such that
$\theta (e_i) = \xi_i$
for any $i \in Q_0$.
It follows also from
Theorem 2.4 of \cite{ES3}
and
Proposition 9.1 of \cite{ES3}
that
$\Omega_{\Lambda^e}^4(\Lambda) \cong {}_1\Lambda_{\sigma}$
in $\mod \Lambda^e$
for some $K$-algebra automorphism $\sigma$ of $\Lambda$.
Hence, we conclude that
$\dim_K \Lambda = \dim_K \Omega_{\Lambda^e}^4(\Lambda)$.
Moreover, by
Proposition~\ref{prop:3.6},
we have
$\Omega_{\Lambda^e}^4(\Lambda) = \Ker S$.
Therefore,
in order to show that $\theta$ induces an isomorphism
$\theta : \Lambda \to \Omega_{\Lambda^e}^4(\Lambda)$
in $\mod \Lambda^e$,
it remains to prove that
$S(\xi_t) = 0$ for any $t \in Q_0$.
Since $K$ has characteristic $2$,
applying
\cite[part (3) on pages 119 and 120]{ES1},
we conclude that
for any vertex $i \in \partial(Q,f)$
and the border loop $\alpha$ at $i$,
the following equalities hold in $\bP_2$
\begin{align*}
    \sum_{b \in \cB_t e_i} b (\alpha \otimes \alpha + e_i \otimes \alpha^2) b^* &= 0 , \\
    \sum_{b \in \cB_t e_i} b (\alpha \otimes \alpha^2 + e_i \otimes \alpha^3) b^* &= 0 , \\
    \sum_{b \in \cB_t e_i} b (\alpha \otimes \alpha^3) b^* &= 0 ,
\end{align*}
because $\alpha^4 = 0$.
Then, for any $t \in Q_0$,
we obtain the equalities
\begin{align*}
  S(\xi_t)
     &= \sum_{b \in \cB_t} S ( b \otimes b^* )
     = \sum_{b \in \cB_t} \sum_{j \in Q_0} S ( b e_j \otimes e_j b^* )
  \\&
     = \sum_{b \in \cB_t} \sum_{j \in Q_0} b S ( e_j \otimes e_j ) b^*
     = \sum_{b \in \cB_t} \sum_{j \in Q_0} b \bar{\psi}_j b^*
     = 0
     .
\end{align*}
This completes the proof that $\Lambda$ is a periodic algebra
of period $4$.
\end{proof}

\bigskip

\section{Proof of Theorem~\ref{th:main2}}
\label{sec:4}

Let $A$ be a basic, indecomposable, symmetric algebra with the Grothendieck
group $K_0(A)$ of rank at least $2$.
Assume that $A$ is socle equivalent to a weighted triangulation algebra
$\Lambda(Q,f,m_{\bullet},c_{\bullet})$.

Let $\Lambda=\Lambda(Q,f,m_{\bullet},c_{\bullet})=KQ/I$,
where $I=I(Q,f,m_{\bullet},c_{\bullet})$.
Since $A/\!\soc (A)$ is isomorphic to $\Lambda /\!\soc(\Lambda)$,
we can assume that these are equal, using an isomorphism
as identification.
We assume that $A$ is symmetric; 
therefore, for each $i \in Q_0$, the module $e_i A$ has
one-dimensional socle which is spanned by some $\omega_i \in e_i A e_i$,
and we fix such an element.
Then, let $\varphi$ be a symmetrizing linear form for $A$,
and then $\varphi(\omega_i)$ is non-zero.
We may assume that $\varphi(\omega_i) = 1$.

Observe also that
$\soc(A) \subset(\rad A)^2$.
If not, then for some $j \in Q_0$ we have $\omega_j \notin(\rad A)^2$.
This means that $e_j A = e_j A e_j$,
which is not possible because $A$ is indecomposable
with $K_0(A)$ of rank at least $2$.
It follows that $A$ and $\Lambda$ have the same Gabriel quiver.
Recall that the Gabriel quiver $Q_{\Lambda}$ of $\Lambda$
is obtained from $Q$ by removing all virtual arrows.
Moreover, for any virtual arrow $\alpha$ in $Q$,
we have in $\Lambda$ the equalities
$\bar{\alpha} f(\bar{\alpha}) = c_{\alpha} A_{\alpha} = c_{\alpha} {\alpha}$,
and this element belongs to the second socle of $\Lambda$, and hence
is a non-zero element of $A$.
We may therefore take $A$ of the form $K Q/L$ for the same quiver $Q$,
and some ideal $L$ of $K Q$ such that
$c_{\alpha} {\alpha} - \bar{\alpha} f(\bar{\alpha}) \in L$
for any virtual arrow $\alpha$ of $Q$.
Since $A_{\alpha}$ is a non-zero element of the socle of  $A/\!\soc (A)$,
we have 
$A_{\alpha}(\rad A)=\soc (e_i A)$,
where $i = s(\alpha)$.
We have also that $A_{\alpha}(\rad A)$ is spanned by 
$A_{\alpha}g^{-1}(\alpha)$
and 
$A_{\alpha}f(g^{-2}(\alpha))$.
We also recall that every $f$-orbit in $Q_1$ is of length $3$ or $1$.
Then, applying the arguments from the proof of 
\cite[Theorem~5.3]{ES8}, we may assume that
$A_{\alpha}g^{-1}(\alpha)\neq0$
and 
$A_{\alpha}f(g^{-2}(\alpha))=0$.
Hence we may assume that
$A_{\alpha}g^{-1}(\alpha) = B_{\alpha}$ in $A$
(and hence is equal to $B_{\alpha}$ in $K Q$).

We have the following relations in $A$:
\begin{enumerate}[\rm (1)]
 \item
	$\alpha f(\alpha) - c_{\bar{\alpha}}A_{\bar{\alpha}} = a_{\alpha}\omega_{s(\alpha)}$,
     	for all arrows
	$\alpha \in Q_1$,
 \item
	$\alpha f(\alpha) g(f(\alpha)) = \lambda_{\alpha}\omega_{s(\alpha)}$,
	for all arrows $\alpha \in Q_1$, except
	$f^2(\alpha)$ is virtual,
	or $f(\bar{\alpha})$ is virtual 
	and $m_{\bar{\alpha}} = 1$, $n_{\bar{\alpha}} = 3$,
 \item
	$\alpha g(\alpha)f(g(\alpha)) = \mu_{\alpha}\omega_{s(\alpha)}$, 
	for all arrows $\alpha \in Q_1$, except
	$f(\alpha)$ is virtual,
	or $f^2({\alpha})$ is virtual 
	and $m_{f(\alpha)} = 1$, $n_{f(\alpha)} = 3$,
\end{enumerate}
for some $a_{\alpha}, \lambda_{\alpha}, \mu_{\alpha} \in K$.

We observe that $a_{\alpha} = 0$ if $t(f(\alpha))\neq s(\alpha)$,
$\lambda_{\alpha} = 0$ if $t(g(f(\alpha)))\neq s(\alpha)$,
and $\mu_{\alpha} = 0$ if $t(f(g(\alpha)))\neq s(\alpha)$.
We also note that if $f(\bar{\alpha})$ is virtual and
$m_{\bar{\alpha}} = 1$, $n_{\bar{\alpha}} = 3$,
then $t(g(f(\alpha)))\neq s(\alpha)$.
Similarly, if $f^2(\alpha)$ is virtual and
$m_{f(\alpha)} = 1$, $n_{f(\alpha)} = 3$,
then $t(f(g(\alpha)))\neq s(\alpha)$.
We will show now that $\lambda_{\alpha} = 0$ and $\mu_{\alpha} = 0$
for the remaining $\alpha \in Q_1$ occurring in (2) and (3).

\smallskip

(a)
Assume that $\alpha$ is an arrow in $Q_1$ such that $f^2(\alpha)$
is not virtual and $t(f(g(\alpha)))= s(\alpha)$.
We note that $\alpha$ is not a loop, because $|Q_0| \geqslant 2$
and $f(\bar{\alpha})\neq g(\bar{\alpha})$.
Then $Q$ contains a subquiver of the form
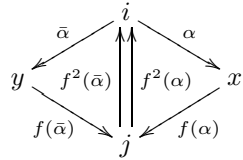
\[
  \xymatrix@R=1pc@C=2.5pc{
   & i \ar[ld]_{\bar{\alpha}} \ar[rd]^{\alpha}  \\
   y  \ar[rd]_{f(\bar{\alpha})}
   && x  \ar[ld]^{f(\alpha)}   \\
   & j  \ar@<.5ex>[uu]^{f^2(\bar{\alpha})}  \ar@<-.5ex>[uu]_{f^2({\alpha})}
  } 
\]
with $f^2(\bar{\alpha}) = g(f(\alpha))$ and $f^2(\alpha) = g(f(\bar{\alpha}))$.
Then there are subpaths $p$ of $A_{\alpha}$ and $q$ of $A_{\bar{\alpha}}$
of length $\geqslant 1$ such that $A_{\alpha} = \alpha p f(\alpha)$
and $A_{\bar{\alpha}} = \bar{\alpha} q f(\bar{\alpha})$.
Since $t(f(\alpha)) =j= t(f(\bar{\alpha}))$ is different from
$s(\alpha) = i = s(\bar{\alpha})$, we obtain the equalities in $A$
\begin{align*}
 \alpha f(\alpha) g\big(f(\alpha)\big)
  &=  c_{\bar{\alpha}}A_{\bar{\alpha}} f^2(\bar{\alpha})
   = c_{\bar{\alpha}} \bar{\alpha} q f(\bar{\alpha})f^2(\bar{\alpha})
   = c_{\bar{\alpha}} c_{g(\bar{\alpha})} \bar{\alpha} q A_{g(\bar{\alpha})}
   = 0
\end{align*}
because $A_{g(\bar{\alpha})}$ is in the second socle of $A$.

Therefore, we have $\alpha f(\alpha) g(f(\alpha))=0$ in $A$
for all arrows $\alpha$ of $Q$ such that $f^2(\alpha)$ is not virtual.

\smallskip

(b)
Assume that $\alpha$ is an arrow in $Q$ such that $f(\alpha)$
is not virtual and $t(f(g(\alpha)))= s(\alpha)$.
Since $|Q_0| \geqslant 2$
we conclude that $\alpha$ is not a loop.
Then $Q$ contains a subquiver of the form
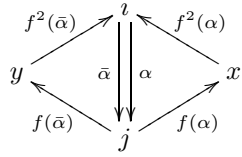
\[
  \xymatrix@R=1pc@C=2.5pc{
   & i \ar@<-.5ex>[dd]_{\bar{\alpha}} \ar@<.5ex>[dd]^{\alpha}  \\
   y  \ar[ru]^{f^2(\bar{\alpha})} 
   && x \ar[lu]_{f^2(\alpha)}   \\
   & j \ar[lu]^{f(\bar{\alpha})} \ar[ru]_{f(\alpha)} 
  } 
\]
with $g(\alpha) = f(\bar{\alpha})$ and $g(\bar{\alpha}) = f(\alpha)$.
Then there is a subpath $u$ of $A_{f(\alpha)}$
of length $\geqslant 1$ such that 
$A_{f(\alpha)} = f(\alpha) u f^2(\alpha)$.
We obtain the equalities in $A$
\begin{align*}
 \alpha g(\alpha) f\big(g(\alpha)\big)
  &=   \alpha f(\bar{\alpha}) f^2(\bar{\alpha})
   = c_{f(\alpha)} \alpha A_{f(\alpha)}
   =  c_{f(\alpha)} \alpha f(\alpha) u f^2({\alpha})
   =  c_{f(\alpha)}  c_{\bar{\alpha}}A_{\bar{\alpha}} u f^2({\alpha})
   = 0
\end{align*}
because $A_{\bar{\alpha}}$ is in the second socle of $A$.

Therefore, we have $\alpha g(\alpha) f(g(\alpha)) =0$ in $A$
for all arrows $\alpha$ of $Q$ such that $f(\alpha)$ is not virtual.

\smallskip

(c)
Assume that  $t(f(\alpha))= s(\alpha) \notin \partial(Q,f)$.
Then $Q$ contains a subquiver of the form
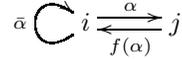
\[
 \xymatrix{
  i \ar@(dl,ul)[]^{\bar{\alpha}} \ar@<+.5ex>[r]^{\alpha}
   & j \ar@<+.5ex>[l]^{f(\alpha)}  
 }
\]
with 
$f(\bar{\alpha}) = \alpha$,
$f^2({\alpha}) = \bar{\alpha}$,
$g(\bar{\alpha}) = \bar{\alpha}$,
and
$g(f(\alpha)) = \alpha$.
Moreover, we have the equality 
$\alpha f(\alpha) - c_{\bar{\alpha}}A_{\bar{\alpha}} = a B_{\alpha}$
for some $a \in K$.
Observe also that $B_{\alpha} = \alpha A_{g(\alpha)}$.
Replacing $f(\alpha)$ by $f(\alpha)^* = f(\alpha) - a A_{g(\alpha)}$,
we obtain 
$\alpha f(\alpha)^* = c_{\bar{\alpha}}A_{\bar{\alpha}}$.
Further, we have 
$A_{g(\alpha)}=v g^{-1}(f(\alpha))f(\alpha)$
for a path $v$ of length $\geqslant 0$,
and $f(g^{-1}(f(\alpha)))f(\alpha) = g(\alpha)$ is not virtual.
This implies that
$A_{g(\alpha)}\bar{\alpha} = v g^{-1}(f(\alpha))f(\alpha)\bar{\alpha} = 0$.
Hence we conclude that
\begin{align*}
 f(\alpha)^* \bar{\alpha} 
  &= f(\alpha) \bar{\alpha}  - a A_{g(\alpha)}\bar{\alpha}
   = f(\alpha) \bar{\alpha} 
   = c_{g(\alpha)} A_{g(\alpha)} .
\end{align*}
We note that the replacement of $f(\alpha)$ by $f(\alpha)^*$
does not change $A_{g(\alpha)}$,
because $A_{g(\alpha)}$ is of length $\geqslant 2$
and belongs to the second socle of $A$.

\smallskip

(d)
Assume that $\alpha$ is a border loop of $(Q,f)$.
Then
$\alpha f(\alpha) = c_{\bar{\alpha}}A_{\bar{\alpha}} + b_{s(\alpha)}\omega_{s(\alpha)}$
for some $b_{s(\alpha)} \in K$.
Moreover, in this case we have 
$B_{\alpha} = B_{\bar{\alpha}}$ 
and
$c_{\alpha} = c_{\bar{\alpha}}$ ,
because $\bar{\alpha} = g(\alpha)$,
so we may take  $\omega_i = B_{\alpha}$.

\smallskip

Summing up, we conclude that we have the border function
$b_{\bullet} : \partial(Q,f) \to K$
such that the ideal $L$ contains the ideal 
$I(Q,f,m_{\bullet},c_{\bullet},b_{\bullet})$,
defining the algebra
$\bar{\Lambda} = \Lambda(S,\vec{T},m_{\bullet},c_{\bullet},b_{\bullet})$.
Moreover, we have
$\dim_K A = \dim_K A/\!\soc(A) + |Q_0|
 = \dim_K \Lambda/\!\soc(\Lambda) + |Q_0|
 = \dim_K \bar{\Lambda}$.
Therefore, $A$ is isomorphic to the algebra $\bar{\Lambda}$.

\section{Reduction to local algebras}
\label{sec:5}

For Theorem~\ref{th:main3}(i) it is necessary to show 
that a socle deformed  weighted surface algebra 
$\Lambda =  \Lambda(S,\vec{T},m_\bullet,c_\bullet,b_\bullet)$ 
with non-zero border  function $b_\bullet$ and at least 
three simple modules in $\mod \Lambda$ is  not isomorphic 
to a (non-deformed) weighted surface algebra 
$\bar{\Lambda} =  \Lambda(S,\vec{T},m_\bullet,c_\bullet)$. 
We will achieve this by showing  that 
the corresponding idempotent algebras of the form 
$e \Lambda e$  and $e \bar{\Lambda} e$ 
are not isomorphic.

Let $K$ be an algebraically closed field, $m \geqslant 2$ a natural
number, and $b \in K$, $c \in K^*$.
Consider the quotient algebra
\[
  A(m,c,b) = K \langle X,Y \rangle / I(m,b,c),
\]
where $I(m,b,c)$ is the ideal generated by the elements
\[
  X^2 - c (YX)^{m-1} Y - b (YX)^m, 
 \qquad
  Y^2, 
 \qquad
  (XY)^m = (YX)^m,
 \qquad
  (XY)^m X,
 \qquad
  (Y X)^m Y.
\]
Then $A(m,c,b)$ is a local symmetric algebra of dimension $4 m$,
with a canonical basis formed by $1$,
the alternating monomials
$(XY)^k$, $(Y X)^k$, $(XY)^k X$, $(Y X)^k Y$, $k \in \{0,\dots,m-1\}$,
and $(XY)^m = (YX)^m$
(see \cite[Theorem~III.1]{E}).

\begin{proposition}
\label{prop:5.1}
Let $K$ be of characteristic $2$, 
$m \geqslant 2$ a positive integer,
and $b,c\in K^*$.
Then the $K$-algebras $A(m,c,0)$ and $A(m,c,b)$
are not isomorphic.
\end{proposition}

\begin{proof}
Suppose that there is an isomorphism of $K$-algebras
$h : A(m,c,0) \to A(m,c,b)$.
Then $h$ is given by the elements
$h(X)=F(X,Y)$ and $h(Y)=G(X,Y)$ in $A(m,c,b)$ of the forms
\begin{align*} 
 F(X,Y)
  &= \sum_{i=0}^{m-1} a_i (XY)^i X
  + \sum_{i=0}^{m-1} b_i (Y X)^i Y
  + \sum_{i=1}^{m-1} c_i (XY)^i
  + \sum_{i=1}^{m-1} d_i (Y X)^i
,
\\
 G(X,Y)
  &= \sum_{i=0}^{m-1} r_i (XY)^i X
  + \sum_{i=0}^{m-1} s_i (Y X)^i Y
  + \sum_{i=1}^{m-1} t_i (XY)^i
  + \sum_{i=1}^{m-1} u_i (Y X)^i
\end{align*}
for some 
$a_i, b_i, c_i, d_i, r_i, s_i, t_i, u_i \in K$ with
$a_0 s_0 \neq b_0 r_0$.
Moreover, since $X^2 = c(Y X)^{m-1} Y$ 
and $Y^2 = 0$ in $A(m,c,0)$, we have 
\[
 F(X,Y)^2 = c\big(G(X,Y)F(X,Y)\big)^{m-1} G(X,Y)
\qquad
 \mbox{ and }
\qquad
 G(X,Y)^2 = 0
 .
\]
Since $m \geqslant 2$,
$XY$, $YX$, $XYX$, $YXY$
are independent elements of $A(m,c,b)$.
The coefficients of $X Y$ and $Y X$
in $F(X,Y)^2$ (respectively, $G(X,Y)^2$)
are equal to $a_0 b_0$ (respectively, $r_0 s_0$).
Hence we have $a_0 b_0 = 0$ and $r_0 s_0 = 0$.
Further, the coefficients of $XYX$ and $YXY$ in $G(X,Y)^2=0$
are equal $r_0(t_1+u_1)$ and $s_0(t_1+u_1)$.
Because $G(X,Y)^2=0$,
we obtain $r_0(t_1+u_1)=0$ and $s_0(t_1+u_1)=0$.
The imposed condition $a_0s_0 \neq b_0 r_0$
forces $r_0 \neq 0$ or $s_0 \neq 0$.
Hence we get $t_1+u_1=0$, or equivalently $t_1=u_1$,
because $K$ is of characteristic $2$.

We claim that the coefficient of $(XY)^m = (YX)^m$ in $G(X,Y)^2$
is equal to $r_0^2b$.
We consider two cases.

\smallskip

(1)
Assume $m$ is odd.
Then the coefficient of $(XY)^m = (YX)^m$ in $G(X,Y)^2$
is equal to
\[
  r_0^2b + \sum_{i+j+1=m} r_i s_j + \sum_{i+j+1=m} s_i r_j
   = r_0^2b ,
\]
because $K$ is of characteristic $2$ and
$X^2 = c(YX)^{m-1} Y + b(YX)^m$ in $A(m,c,b)$.

\smallskip

(2)
Assume $m$ is even, say $m = 2 k$ for some positive natural number $k$.
Then the coefficient of $(XY)^m = (YX)^m$ in $G(X,Y)^2$
is equal to
\[
  r_0^2b + \sum_{i+j+1=m} r_i s_j + \sum_{i+j+1=m} s_i r_j  + t_k^2 + u_k^2
   = r_0^2b + (t_k + u_k)^2 .
\]
Moreover, the coefficient of $(XY)^k X=X(Y X)^k$ 
(respectively, $(Y X)^k Y=Y (X Y)^k$) in $G(X,Y)^2$
is equal to $r_0 (t_k + u_k)$ (respectively, $s_0 (t_k + u_k)$).
Since $G(X,Y)^2=0$
and $r_0 \neq 0$ or $s_0 \neq 0$, we get
$t_k + u_k = 0$, and hence $(t_k + u_k)^2 = 0$.

Therefore, the coefficient of $(XY)^m = (YX)^m$ in $G(X,Y)^2 = 0$
is equal to $r_0^2b$, and hence $r_0 = 0$, because $b\neq0$.
In particular, we have $s_0 \neq 0$.
Then the conditions $a_0 s_0 \neq b_0 r_0$ and $a_0 b_0 = 0$ 
imply $a_0 \neq 0$ and $b_0 = 0$.

Using $b_0 = 0$ and $r_0 = 0$, we conclude that 
the coefficient of $(XY)^m = (YX)^m$ in $(G(X,Y)F(X,Y))^{m-1} G(X,Y)$
is equal to
\[
  (s_0 a_0)^{m-1} u_1 + a_0 t_1 (a_0 s_0)^{m-2} s_0
  = (s_0 a_0)^{m-1} (t_1 + u_1) = 0 ,
\]
because $t_1 = u_1$.
Hence, by the equality
$F(X,Y)^2 = c(G(X,Y)F(X,Y))^{m-1} G(X,Y)$,
we obtain that the coefficient of $(XY)^m = (YX)^m$ in $F(X,Y)^2$
is equal to $0$.

Now, if $m \geq 3$, then applying arguments as in (1) and (2),
we conclude that the coefficient of $(XY)^m = (YX)^m$ in $F(X,Y)^2$
is equal to $a_0^2 b$, which is a contradiction
because $a_0\neq0$ and $b\neq0$.

Finally, assume that $m=2$.
Then we have the equality
$F(X,Y)^2 = c G(X,Y)F(X,Y) G(X,Y)$
with $F(X,Y)$ and $G(X,Y)$ of the forms
\begin{align*} 
 F(X,Y)
  &= a_0 X + a_1 X Y X + b_1 Y X Y + c_1 X Y + d_1 Y X + d_2 (Y X)^2
,
\\
 G(X,Y)
  &= s_0 Y + r_1 X Y X + s_1 Y X Y + t_1 X Y + u_1 Y X + u_2 (Y X)^2
,
\end{align*}
because $b_0 = 0$ and $r_0 = 0$.
Then
\begin{align*} 
 F(X,Y)^2
  &= a_0 (c_1+d_1) X Y X + a_0^2 c Y X Y + ( a_0^2 b + c_1^2 + d_1^2) (X Y)^2
,
\\
c G(X,Y)F(X,Y)G(X,Y)
  &= c a_0 s_0^2 Y X Y +  c a_0 s_0 (t_1 + u_1)  (X Y)^2
  = a_0 s_0^2 c Y X Y 
,
\end{align*}
because $t_1 + u_1 = 0$.
Hence we obtain the equalities
\[
 a_0 (c_1 + d_1) = 0,
 \qquad
 a_0^2 c = a_0 s_0^2 c,
 \qquad
 a_0^2 b + (c_1 + d_1)^2 = 0.
\]
Since $a_0 \neq 0$, we have $c_1 + d_1 = 0$.
But then $a_0^2 b = 0$,
which is again a contradiction.

This finishes the proof.
\end{proof}

\begin{proposition}
\label{prop:5.2}
Let $\Lambda = \Lambda(Q,f,m_{\bullet},c_{\bullet},b_{\bullet})$
be a socle deformed weighted triangulation algebra,
$\alpha$ a border loop of $(Q,f)$, 
$m = m_{\cO(\alpha)}$, 
$c = c_{\cO(\alpha)}$, 
$b = b_{s(\alpha)}$.
Assume that $m \geqslant 2$
and $f(\bar{\alpha})$ is not a virtual loop.
Then the local algebra
$A = e_{s(\alpha)} \Lambda e_{s(\alpha)}$
is isomorphic to $A(m,c,b)$.
\end{proposition}

\begin{proof}
Let $i=s(\alpha)$.
Since $f(\alpha) = \alpha$ and $|Q_0| \geqslant 2$,
we have $\bar{\alpha} = g(\alpha)$ and $n_{\alpha} \geqslant 3$.
Let $X=\alpha$ and $Y=g(\alpha)\dots g^{n_{\alpha}-1}(\alpha)$.
Then
\[
 X^2 = \alpha^2 
  = c_{\bar{\alpha}} A_{\bar{\alpha}} +b_i B_{\bar{\alpha}}
  = c A_{g(\alpha)} + b B_{g(\alpha)} 
  = c (Y X)^{m-1} Y + b (Y X)^m 
  .
\]

We claim that $Y^2 = 0$.
Let 
$\beta = g(\alpha)$,
$\gamma = f(\beta)$,
and $\sigma = f(\gamma)$.
We note that $\beta, \gamma, \sigma$ are pairwise different.
We consider two cases.

\smallskip

(1)
Assume $f^2(\sigma)=\gamma$ is not virtual. 
Then we have
\[
  Y^2=\big(g(\alpha)\dots g^{n_{\alpha}-1}(\alpha)\big)^2
   = g(\alpha)\dots \sigma f(\sigma) 
      g\big(f(\sigma)\big) \dots g^{n_{\alpha}-1}(\alpha)
   = 0
   .
\]

\smallskip

(2)
Assume $f^2(\sigma)=\gamma$ is virtual. 
Then, by assumption, $\gamma$ is not a loop, and 
$(Q,f)$ contains a subquiver of the form
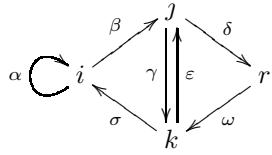
\[
  \xymatrix@R=1pc{
   & j \ar[rd]^{\delta} \ar@<-.5ex>[dd]_{\gamma}  \\
   i   \ar@(ld,ul)^{\alpha}[]  \ar[ru]^{\beta} 
   && r  \ar[ld]^{\omega}   \\
   & k \ar[lu]^{\sigma}  \ar@<-.5ex>[uu]_{\varepsilon}
  } 
\]
with $\delta=f(\varepsilon)$, $\omega = f(\delta)$, $\varepsilon = f(\omega)$,
and we have
\[
  Y^2=\big(g(\alpha)\dots g^{n_{\alpha}-1}(\alpha)\big)^2
   = g(\alpha)\dots \sigma \beta \delta g(\delta) \dots g^{n_{\alpha}-1}(\alpha)
   = c_{\varepsilon} g(\alpha)\dots \varepsilon
       \delta g(\delta) \dots g^{n_{\alpha}-1}(\alpha)
   = 0
   ,
\]
because $\delta = f(\varepsilon)$,
and $f^2(\varepsilon)=\omega$
is not virtual.

Finally, we observe that
\begin{align*} 
  (XY)^m X
  &= \big(\alpha g(\alpha)\dots g^{n_{\alpha}-1}(\alpha)\big)^{m_{\alpha}}\alpha 
    = B_{\alpha} \alpha = 0 
,
\\
  (Y X)^m Y
  &= \big(g(\alpha)\dots g^{n_{\alpha}-1}(\alpha) \alpha\big)^{m_{\alpha}}
        g(\alpha)\dots g^{n_{\alpha}-1}(\alpha)
     = B_{g(\alpha)} g(\alpha)\dots g^{n_{\alpha}-1}(\alpha)
     = 0 
.
\end{align*}
\end{proof}

\section{Reduction to algebras with two vertices}
\label{sec:6} 
This deals with technical reductions 
towards the proof of Theorem~\ref{th:main3}. 

Let $r \geqslant 1$ be a natural number and $b\in K$, $c \in K^*$.
Consider the algebra $B(r,c,b)$
given by the quiver $Q$
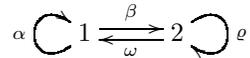
\[
  \xymatrix{
    1
    \ar@(ld,ul)^{\alpha}[]
    \ar@<.5ex>[r]^{\beta}
    & 2
    \ar@<.5ex>[l]^{\omega}
    \ar@(ru,dr)^{\varrho}[]
  } 
\]
and the relations:
\begin{align*}
&&
  \alpha^2 &= c \beta \omega + b \beta \omega \alpha,&
  \alpha \beta \omega & = \beta \omega \alpha,&
  \omega \alpha \beta &= \varrho^r, &
  \omega \beta &= 0, &
  \beta \varrho &=0, &
  \varrho \omega &= 0.
&&
\end{align*}
We observe that we have the following additional relations:
\begin{align*}
   \beta \omega \alpha \beta & = \beta \varrho^r = 0, 
\\
  \alpha^2 \beta &= c \beta \omega \beta + b \beta \omega \alpha \beta = 0,
\\
  \omega \alpha^2 &= c \omega \beta \omega + b \omega \beta \omega \alpha = 0,
\\
  \alpha \beta \omega \alpha &= \beta \omega  \alpha^2  =
	c\beta \omega \beta  \omega
      + b \beta \omega \beta \omega \alpha = 0, 
\\
  c \beta \omega \alpha &= \alpha^3 =  c \alpha \beta \omega  .
\end{align*}
In particular, $B(r,c,b)$ has a basis
\[
e_1,\, e_2,\, 
\alpha,\, \alpha \beta,\, \beta,\, \beta \omega,\, 
\omega,\, \omega \alpha,\,
\varrho,\,\dots,\,\varrho^{r-1},\,\,
\varrho^r=\omega\alpha\beta,\,\, \beta\omega\alpha.
\]

\begin{lemma}
\label{lem:6.1}
Let $K$ be of characteristic $2$, 
$r \geqslant 1$ a natural number,
and $b,c\in K^*$.
Then the $K$-algebras $B(r,c,0)$ and $B(r,c,b)$
are not isomorphic.
\end{lemma}

\begin{proof}
Assume there exists an isomorphism of $K$-algebras
$h : B(r,c,0) \to B(r,c,b)$ such that 
$h(e_1) = e_1$ and $h(e_2) = e_2$.
Then we have
\begin{align*}
h(\alpha) &= r_1 \alpha + r_2 \beta \omega + r_3 \beta \omega \alpha,
\\
h(\beta) &= s_1 \beta + s_2 \alpha \beta,
\\
h(\omega) &= t_1 \omega + t_2 \omega \alpha,
\\
h(\varrho) &= \sum_{i=1}^r u_i \varrho^i ,
\end{align*}
for some $r_1,s_1,t_1,u_1\in K^*$ and $r_2,r_3,s_2,t_2,u_2,\dots,u_r\in K$.
Since $\alpha^2 = c \beta \omega$ in $B(r,c,0)$, we have
\[
  h(\alpha)^2 = h(\alpha^2) = h(c\beta\omega) = c h(\beta) h(\omega)
\]
in $B(r,c,b)$.
Since $K$ is of characteristic $2$, 
using the relations in $B(r,c,b)$ described above, we obtain that
\begin{align*}
  h(\alpha)^2 
  &= \big(r_1 \alpha + r_2 \beta \omega + r_3 \beta \omega \alpha \big)^2
  = r_1^2 c  \beta \omega +  r_1^2 b  \beta \omega \alpha
,
\\
c h(\beta) h(\omega) &=
  c (s_1 \beta + s_2 \alpha \beta) (t_1 \omega + t_2 \omega \alpha)
  = c s_1 t_1 \beta \omega + c(s_1 t_2 + s_2 t_1) \beta \omega \alpha 
.
\end{align*}
Hence $r_1^2 = s_1 t_1$ and $r_1^2 b = c(s_1 t_2 + s_2 t_1)$.

Further, we have the equalities
\[
 0 = h(\omega \beta) = h(\omega) h(\beta) 
 = (t_1 \omega + t_2 \omega \alpha) (s_1 \beta + s_2 \alpha \beta)
 = (s_1 t_2 + s_2 t_1) \omega \alpha \beta
,
\]
and hence $s_1 t_1 + s_2 t_1 = 0$.
But then $r_1^2 b = 0$,
which contradicts $r_1 \neq 0$ and $b \neq 0$.
\end{proof}

\begin{proposition}
\label{prop:6.2}
Let $\Lambda = \Lambda(Q,f,m_{\bullet},c_{\bullet},b_{\bullet})$
be a socle deformed weighted triangulation algebra,
$\alpha$ a border loop of $(Q,f)$, 
$c = c_{\cO(\alpha)}$, 
$b = b_{s(\alpha)}$,
and $e = e_{s(\alpha)}+e_{t(g(\alpha))}$.
Assume that 
$|Q_0|\geqslant 3$,
$m_{\cO(\alpha)} = 1$, 
and
$\cO(\alpha) \neq \cO(f(\bar{\alpha}))$.
Then the idempotent algebra $e \Lambda e$
is isomorphic to $B(r,c,b)$ for some $r\geqslant1$.
\end{proposition}

\begin{proof}
Let 
$\beta = g(\alpha)$,
$\gamma = f(\beta)$,
$\sigma = f(\gamma)$,
$i = s(\alpha)$,
$j = s(\gamma)$,
and $k = s(\sigma)$.
It follows from the assumption that 
$|Q_0|\geqslant4$
and
$\cO(\alpha) = \cO(\beta) = \cO(\sigma) \neq \cO(\gamma)$. 
In particular, we have 
$n_{\gamma} = |\cO(\gamma)| \geqslant 2$. 
Then $e = e_i+e_j$
and
$B = e \Lambda e$
is given by the quiver $Q$
\[
  \xymatrix{
    i
    \ar@(ld,ul)^{\alpha}[]
    \ar@<.5ex>[r]^{\beta}
    & j
    \ar@<.5ex>[l]^{\omega}
    \ar@(ru,dr)^{\varrho}[]
  } 
\]
with
$\omega = g(\beta)\dots g^{n_{\alpha}-1}(\alpha)
  = g(\beta)\dots \sigma$,
$\varrho = \gamma g(\gamma)\dots g^{n_{\gamma}-1}(\gamma)$,
and the relations induced from $\Lambda$.
Let 
$\delta = g(\beta)$,
$\varepsilon = g^{-1}(\sigma)$,
$\xi = g^{-1}(\gamma)$.
Observe that
\[
  \alpha^2 = c_{\beta} A_{\beta} + b_{s(\alpha)} B_{\beta}
  = c \beta \omega + b \beta \omega \alpha ,
\]
because $m_{\beta} = m_{\alpha} = 1$.
Moreover, we have
\[
  c \alpha \beta \omega
  = c_{\alpha} B_{\alpha} = c_{\bar{\alpha}} B_{\bar{\alpha}} 
  = c_{\beta} B_{\beta} 
  = c \beta \omega \alpha ,
\]
and hence
$\alpha \beta \omega = \beta \omega \alpha$.
Further, we have
\[
  c \omega \alpha \beta
  = c_{g(\beta)} B_{g(\beta)} 
  = c_{\delta} B_{\delta} 
  = c_{\bar{\delta}} B_{\bar{\delta}} 
  = c_{\gamma} B_{\gamma} 
  = c_{\gamma} \varrho^{m_{\gamma}} ,
\]
and hence
$\omega \alpha \beta = c^{-1} c_{\gamma} \varrho^{m_{\gamma}}$.
Taking $r = m_{\gamma}$ and replacing
$\varrho$ by $a\varrho$ with $a^r = c c_j^{-1}$,
we may assume that $\omega \alpha \beta = \varrho^r$.

We claim that the relations
$\omega \beta = 0$,
$\beta \varrho = 0$,
$\varrho \omega = 0$
also hold.
We consider two cases.

\smallskip

(1)
Assume $n_{\gamma}\geqslant 3$.
Then we have
\[
 \beta \varrho
  = \beta  \gamma g(\gamma)\dots g^{n_{\gamma}-1}(\gamma)
  = \beta f(\beta) g\big(f(\beta)\big)\dots g^{n_{\gamma}-1}(\gamma)
  = 0 ,
\]
because $f^2(\beta) = \sigma$ is not virtual.
Further,
\[
  \omega \beta
  = \delta  g(\delta)\dots \varepsilon \sigma \beta
  = \delta  g(\delta)\dots \varepsilon g(\varepsilon) f\big(g(\varepsilon)\big)
  = 0 ,
\]
because $f(\varepsilon)$ is not virtual by $n_{\gamma}\geqslant 3$.
Observe that
\[
  \varrho \omega
  = \gamma g(\gamma)\dots g^{n_{\gamma}-1}(\gamma)
     g(\beta)g^2(\beta)\dots g^{n_{\alpha}-1}(\alpha)
  =  \gamma g(\gamma)\dots \xi \delta g(\delta) \dots \sigma ,
\]
with $f(\xi) = \delta$.
Hence $\varrho \omega = 0$,
if $f^2(\xi)$ is not virtual.
Assume $f^2(\xi)$ is virtual.
Then $(Q,f)$ contains a subquiver of the form
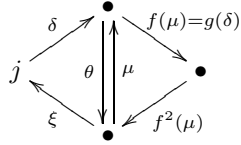
\[
  \xymatrix@R=1pc{
   & \bullet \ar[rd]^{f(\mu) = g(\delta)} \ar@<-.5ex>[dd]_{\theta}  \\
   j   \ar[ru]^{\delta} 
   && \bullet \ar[ld]^{f^2(\mu)}   \\
   & \bullet \ar[lu]^{\xi}  \ar@<-.5ex>[uu]_{\mu}
  } 
\]
and
$\varrho \omega
  = \gamma \dots \xi \delta g(\delta) g^2(\delta) \dots \sigma 
  = c_{\mu} \gamma \dots \mu g(\delta) g^2(\delta) \dots \sigma 
  = 0$,
because $f^2(\mu)$ is not virtual.

\smallskip

(2)
Assume $n_{\gamma}= 2$.
Then $(Q,f)$ contains a subquiver 
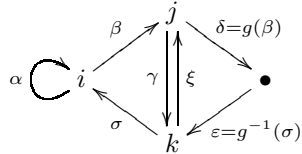
\[
  \xymatrix@R=1pc{
   & j \ar[rd]^{\delta=g(\beta)} \ar@<-.5ex>[dd]_{\gamma}  \\
   i   \ar@(ld,ul)^{\alpha}[]  \ar[ru]^{\beta} 
   && \bullet \ar[ld]^{\varepsilon=g^{-1}(\sigma)}   \\
   & k \ar[lu]^{\sigma}  \ar@<-.5ex>[uu]_{\xi}
  } 
\]
and $\varrho = \gamma \xi$.
Then we have  the equalities
\begin{align*}
&&
 \beta \varrho
  &= \beta  \gamma \xi = \beta f(\beta) g\big(f(\beta)\big) = 0,
&&
     \mbox{because $f^2(\beta) = \sigma$ is not virtual},
&&
\\
&&
  \omega \beta
  &= \delta  g(\delta)\dots \varepsilon \sigma \beta
   = \delta  g(\delta)\dots \varepsilon g(\varepsilon) f\big(g(\varepsilon)\big) = 0,
&&
     \mbox{because $f^2(\varepsilon) = \delta$ is not virtual},
&&
\\
&&
  \varrho \omega
  &= \gamma \xi \delta  g(\delta)\dots  \sigma
   = \gamma \xi f(\xi) g\big(f(\xi)\big)\dots  \sigma = 0,
&&
     \mbox{because $f^2(\xi) = \varepsilon$ is not virtual}.
&&
\end{align*}
\end{proof}

\section{Reduction to algebras with three vertices}
\label{sec:7}

We deal with the technical reduction in this case, towards the proof of
Theorem 1.3.

Let $\Delta$ be the quiver
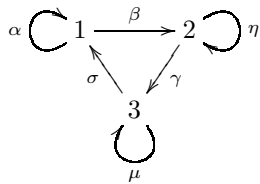
\[
  \xymatrix@C=.8pc@R=1.5pc{
     1 \ar@(dl,ul)[]^{\alpha} \ar[rr]^{\beta} &&
     2 \ar@(ur,dr)[]^{\eta} \ar[ld]^{\gamma} \\
     & 3 \ar@(dr,dl)[]^{\mu} \ar[lu]^{\sigma}}
\]
and $b_{\bullet}, c_{\bullet} : Q_0 \to K$ two functions.
We consider the algebra $D(b_{\bullet}, c_{\bullet})$ 
given by the quiver  $\Delta$ and the relations:
\begin{align*}
&&
\beta \gamma &= c_1 \alpha \beta \eta \gamma \mu,
&
\alpha \beta \eta \gamma \mu \sigma &=  \beta \eta \gamma \mu \sigma \alpha,
&
\alpha^2 &= c_1  \beta \eta \gamma \mu \sigma 
   + b_1  \beta \eta \gamma \mu \sigma \alpha,
&
 \beta \gamma \mu &= 0,
&
 \mu \sigma \beta &= 0,
&
 \gamma \sigma \alpha &= 0,
 &&
\\
&&
\gamma \sigma &= c_2 \eta \gamma \mu \sigma \alpha,
&
\eta \gamma \mu \sigma \alpha \beta &= \gamma \mu \sigma \alpha \beta \eta ,
&
\eta^2 &= c_2   \gamma \mu \sigma \alpha \beta  
   + b_2   \gamma \mu \sigma \alpha \beta \eta,
&
 \alpha \beta \gamma &= 0,
&
 \sigma \beta \eta &= 0,
&
 \eta \gamma \sigma &= 0,
 &&
\\
&&
\sigma \beta &= c_3 \mu \sigma \alpha \beta \eta,
&
\mu \sigma \alpha \beta \eta \gamma &=  \sigma \alpha \beta \eta \gamma \mu,
&
\mu^2 &= c_3  \sigma \alpha \beta \eta \gamma 
   + b_3 \sigma \alpha \beta \eta \gamma \mu,
&
 \alpha^2 \beta &= 0,
&
 \sigma \alpha^2 &= 0.
 &&
\end{align*}
We observe that
\begin{align*}
\eta^2 \gamma &= c_2   \gamma \mu \sigma \alpha \beta \gamma 
   + b_2  \gamma \mu \sigma \alpha \beta \eta \gamma
   = b_2  \eta \gamma \mu \sigma \alpha \beta \gamma
   = 0,
\\
\beta \eta^2 &= c_2   \beta \gamma \mu \sigma \alpha \beta  
   + b_2   \beta \gamma \mu \sigma \alpha \beta \eta
   = 0,
\\
\mu^2 \sigma &= c_3  \sigma \alpha \beta \eta \gamma \sigma 
   + b_3 \sigma \alpha \beta \eta \gamma \mu \sigma
   = b_3 \mu \sigma \alpha \beta \eta \gamma \sigma
   = 0,
\\
\gamma \mu^2 &= c_3  \gamma \sigma \alpha \beta \eta \gamma 
   + b_3 \gamma \sigma \alpha \beta \eta \gamma \mu
   = 0.
\end{align*}
Therefore, 
 $D(b_{\bullet}, c_{\bullet})$ 
is a $K$-algebra of dimension $36$.

\begin{lemma}
\label{lem:7.1}
Let $K$ be of characteristic $2$, and
$A= D(b'_{\bullet}, c'_{\bullet})$ 
and
$B= D(b_{\bullet}, c_{\bullet})$ 
with 
$b'_1 = 0$,
$c'_1 \neq 0$,
$b_1 \neq 0$,
$c_1 \neq 0$.
Then there is no $K$-algebra isomorphism 
$h : A \to B$ such that 
$h(e_i) = e_i$ for any $i \in \{1,2,3\}$.
\end{lemma}

\begin{proof}
Suppose there exists a $K$-algebras isomorphism 
$h : A \to B$ with 
$h(e_i) = e_i$ for any $i \in \{1,2,3\}$.
Then  there exist elements
$r_1,s_1,t_1\in K^*$ and $r_i,s_i,t_i\in K$, $i \in \{2,3,4\}$,
such that
\begin{align*}
h(\beta) &= r_1 \beta + r_2 \alpha \beta + r_3 \beta \eta
        + r_4 \alpha \beta \eta,
\\
h(\gamma) &= s_1 \gamma + s_2 \eta \gamma  + s_3 \gamma \mu
      	+ s_4 \eta \gamma \mu,
\\
h(\sigma) &= t_1 \sigma + t_2 \mu \sigma + t_3 \sigma \alpha 
	+ t_4 \mu \sigma \alpha.
\end{align*}
We observe that we have in $B$ the equalities
\begin{align*}
\alpha^3 &= c_1 \alpha \beta \eta \gamma \mu \sigma
	   + b_1 \alpha \beta \eta \gamma \mu \sigma \alpha
	= c_1 \alpha \beta \eta \gamma \mu \sigma
	   + b_1 \alpha^2 \beta \eta \gamma \sigma
	= c_1 \alpha \beta \eta \gamma \mu \sigma
	= c_1 \beta \eta \gamma \mu \sigma \alpha.
\end{align*}
Since $c_1 \neq 0$, we conclude that
$\alpha^3 \neq 0$.
Moreover, $\alpha^4 = 0$ because $\sigma \alpha^2 = 0$.
Hence
\begin{align*}
h(\alpha) &= u_1 \alpha +  u_2 \alpha^2 +  u_3 \alpha^3,
\end{align*}
for some
$u_1\in K^*$ and $u_2,u_3\in K$.
Moreover, there exist 
$v_1,w_1\in K^*$ and $v_2,v_3,w_2,w_3\in K$
such that
\begin{align*}
h(\eta) &= v_1 \eta + v_2 \gamma \mu \sigma \alpha \beta
	+ v_3  \gamma \mu  \sigma \alpha \beta  \eta,
\\
h(\mu) &= w_1 \mu + w_2 \sigma \alpha \beta \eta \gamma 
	+ w_3\sigma \alpha \beta \eta \gamma \mu .
\end{align*}
Since $b'_1 = 0$ and $K$ is of characteristic $2$,
we conclude that the following equalities hold in $B$
\begin{align*}
u_1^2 &  c_1 \beta \eta \gamma \mu \sigma 
 +  u_1^2 b_1 \beta \eta \gamma \mu \sigma \alpha
 = u_1^2  \alpha^2 = \big(u_1 \alpha + u_2 \alpha^2 + u_3 \alpha^3\big)^2
\\
 &= h(\alpha)^2 = h(\alpha^2)
 = h(c'_1 \beta \eta \gamma \mu \sigma)
 = c'_1 h(\beta) h(\eta) h(\gamma) h(\mu) h(\sigma)
\\
 &= c'_1 \big( 
   r_1 v_1 s_1 w_1 t_1  \beta \eta \gamma \mu \sigma
   + r_2 v_1 s_1 w_1 t_1  \alpha \beta \eta \gamma \mu \sigma
   + r_1 v_1 s_1 w_1 t_3  \beta \eta \gamma \mu \sigma \alpha
 \big)
\\
 &= c'_1 r_1 v_1 s_1 w_1 t_1  \beta \eta \gamma \mu \sigma
   + c'_1  v_1 s_1 w_1 (r_2 t_1 + r_1 t_3) \beta \eta \gamma \mu \sigma \alpha
,
\end{align*}
and hence 
$u_1^2 c_1 = c'_1 r_1  v_1 s_1 w_1 t_1$
and
$u_1^2 b_1 = c'_1  v_1 s_1 w_1 (r_2 t_1 + r_1 t_3)$.
In particular, we obtain $r_2 t_1 + r_1 t_3 \neq 0$,
because
$u_1,b_1,c'_1,v_1,s_1,w_1\in K^*$.
On the other hand,
we have the following equalities in $B/(\rad B)^4$
\begin{align*}
0 + (\rad B)^4 
 &= h (c'_1  \mu  \sigma \alpha \beta  \eta) + (\rad B)^4 
 = h (\sigma \beta) + (\rad B)^4 
 = h (\sigma) h (\beta) + (\rad B)^4 
\\&
 = 
 (t_1 \sigma + t_2 \mu \sigma + t_3 \sigma \alpha + t_4 \mu \sigma \alpha)
 (r_1 \beta + r_2 \alpha \beta + r_3 \beta \eta + r_4 \alpha \beta \eta)
 + (\rad B)^4 
\\&
 =(t_1 r_2 + t_3 r_1) \sigma \alpha \beta + (\rad B)^4 
,
\end{align*}
because 
$\sigma \beta = c'_1 \mu  \sigma \alpha \beta  \eta \in (\rad B)^5$.
Hence $r_2 t_1 + r_1 t_3 = 0$, a contradiction.
This proves the claim.
\end{proof}

\begin{proposition}
\label{prop:7.2}
Let $\Lambda = \Lambda(Q,f,m_{\bullet},c_{\bullet},b_{\bullet})$
be a socle deformed weighted triangulation algebra,
$\alpha$ a border loop of $(Q,f)$, 
and assume that
$\cO(\alpha) = \cO(f(\bar{\alpha}))$,
$m_{\cO(\alpha)} = 1$, 
and take $e = e_{s(\alpha)}+e_{t(g(\alpha))}+e_{t(f(g(\alpha)))}$.
Then the idempotent algebra $e \Lambda e$
is isomorphic to an algebra $ D(b_{\bullet}, c_{\bullet})$ 
with
$c_1 = c_{\alpha}$
and 
$b_1 = b_{s(\alpha)}$.
\end{proposition}

\begin{proof}
Let 
$\beta = g(\alpha)$,
$\gamma = f(\beta)$,
$\sigma = f(\gamma)$,
$1 = s(\alpha)$,
$2 = t(\beta)$,
$3 = t(\gamma)$,
$\eta = g(\beta)\dots g^{-1}(\gamma)$,
$\mu = g(\gamma)\dots g^{-1}(\sigma)$.
Since 
$f(\alpha) = \alpha$,
$\beta = \bar{\alpha}$,
$\gamma = f(\bar{\alpha})$,
and
$\cO(\alpha) = \cO(\gamma)$,
we conclude that
$|Q_0|\geqslant3$
and
$\beta,\gamma,\sigma$
are pairwise different.
Then $e \Lambda e$
is given by the quiver 
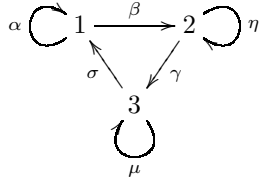
\[
  \xymatrix@C=.8pc@R=1.5pc{
     1 \ar@(dl,ul)[]^{\alpha} \ar[rr]^{\beta} &&
     2 \ar@(ur,dr)[]^{\eta} \ar[ld]^{\gamma} \\
     & 3 \ar@(dr,dl)[]^{\mu} \ar[lu]^{\sigma}}
\]
and the relations induced from $\Lambda$.
Taking
$c_1 = c_{\alpha}$
and
$b_1 = b_{s(\alpha)}$,
we conclude that
\begin{align*}
 \alpha^2 &= c_{\beta} A_{\beta} +  b_{s(\alpha)} B_{\beta}
  = c_1 \beta  \eta \gamma \mu \sigma 
     + b_1 \beta  \eta \gamma \mu \sigma \alpha
,\\
 \beta \gamma &=  c_{\alpha} A_{\alpha} 
  = c_1 \alpha \beta \eta \gamma \mu 
,\\
 \gamma \sigma &=  c_{g(\beta)} A_{g(\beta)} 
  = c_1 \eta \gamma \mu \sigma \alpha
,\\
\sigma  \beta &=  c_{g(\gamma)} A_{g(\gamma)} 
  = c_1 \mu \sigma \alpha \beta \eta
 ,\\
 \beta \gamma \mu &=  
  \beta f(\beta) g\big(f(\beta)\big)\dots g^{-1}(\sigma)
  = 0, 
\mbox{ because $f^2(\beta) = \sigma$ is not virtual}
,\\
 \gamma \sigma \alpha &=  
  \gamma f(\gamma) g\big(f(\gamma)\big) = 0, 
\mbox{ because $f^2(\gamma) = \beta$ is not virtual}
,
\end{align*}
\begin{align*}
\sigma \beta \eta &=  
  \sigma f(\sigma) g(\beta)\dots g^{-1}(\gamma)
  = 0, 
\mbox{ because $f^2(\sigma) = \gamma$ is not virtual}
,\\
 \mu \sigma \beta &=  
  g(\gamma)\dots g^{-1}(\sigma) \sigma f(\sigma)
  = 0, 
\mbox{ because $f\big(g^{-1}(\sigma))$ is not virtual}
,\\
 \alpha \beta \gamma &=  
  \alpha g(\alpha) f\big(g(\alpha)\big)
  = 0, 
\mbox{ because $f(\alpha) = \alpha$ is not virtual}
,\\
  \eta \gamma \sigma &=  
  g(\beta) \dots g^{-1}(\gamma) \gamma f(\gamma)
  = 0, 
\mbox{ because $f\big(g^{-1}(\gamma))$ is not virtual}
,\\
 \alpha^2 \beta &=  
  \alpha f(\alpha) g\big(f(\alpha)\big) = 0, 
\mbox{ because $f^2(\alpha) = \alpha$ is not virtual}
,\\
 \sigma \alpha^2 &=  
  \sigma g(\sigma) f\big(g(\sigma)\big)
  = 0, 
\mbox{ because $f(\sigma) = \beta$ is not virtual}
.
\end{align*}
Further, since 
$c_{\alpha} B_{\alpha} = c_{\beta} B_{\beta}$,
$c_{g(\beta)} B_{g(\beta)} = c_{\gamma} B_{\gamma}$,
$c_{g(\gamma)} B_{g(\gamma)} = c_{\sigma} B_{\sigma}$,
and 
$c_1 = c_{\alpha} = c_{\beta} = c_{g(\beta)} = c_{\gamma} = c_{g(\gamma)} = c_{\sigma}$,
we obtain the equalities
\begin{align*}
&&
  \alpha \beta  \eta \gamma \mu \sigma 
  &=  \beta  \eta \gamma \mu \sigma \alpha,
&
   \eta \gamma \mu \sigma \alpha \beta
  &= \gamma \mu \sigma \alpha  \beta  \eta,
&
  \mu \sigma  \alpha \beta  \eta \gamma
  &= \sigma \alpha  \beta  \eta \gamma \mu.
&&
\end{align*}
We will prove now that
\[
 \eta^2 = c_2 \gamma \mu \sigma \alpha \beta 
               + b_2 \gamma \mu \sigma \alpha \beta \eta
\]
for some $c_2,b_2 \in K$.
If $\eta = g(\beta) = g^{-1}(\gamma)$ then
$f(\eta) = \eta$, and taking
$c_2 = c_{\gamma}$, $b_2 = b_{s(\eta)}$, we have
\[
 \eta^2 = c_{\gamma} A_{\gamma} + b_{s(\eta)} B_{\gamma}
      = c_2 \gamma \mu \sigma \alpha \beta 
          + b_2 \gamma \mu \sigma \alpha \beta \eta .
\]

Assume that $g(\beta) \neq g^{-1}(\gamma)$,
and set $\delta = g(\beta)$, $\xi  = g^{-1}(\gamma)$.
Note that then $\delta = f(\xi)$.
If $g(\delta) \neq \xi$, we have
\[
 \eta^2 = \delta g(\delta) \dots \xi \delta g(\delta) \dots \xi .
\]
We claim that then $\eta^2 = 0$,
so we may take $c_2 = 0$, $b_2 = 0$. 
If $f^2(\xi)$ is not virtual, then
$\xi \delta g(\delta) = \xi f(\xi) g(f(\xi)) = 0$,
and hence $\eta^2 = 0$.
Assume that $f^2(\xi)$ is virtual.
Then $(Q,f)$ contains a subquiver of the form
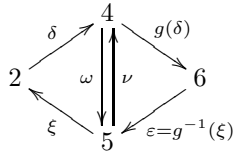
\[
  \xymatrix@R=1pc{
   & 4 \ar[rd]^{g(\delta)} \ar@<-.5ex>[dd]_{\omega}  \\
  2   \ar[ru]^{\delta} 
   && 6 \ar[ld]^{\varepsilon=g^{-1}(\xi)}   \\
   & 5 \ar[lu]^{\xi}  \ar@<-.5ex>[uu]_{\nu}
  } 
\]
with $g^2(\delta) \neq \varepsilon$.
Then we obtain
\[
 \eta^2 = \delta g(\delta) \dots \xi \delta g(\delta) g^2(\delta) \dots \varepsilon \xi
  = c_{\nu} \delta g(\delta) \dots \nu  g(\delta) g^2(\delta) \dots \varepsilon \xi
  = 0 ,
 \]
because $g(\delta) = f(\nu)$ and $f^2(\nu) = \varepsilon$ is not virtual.

Assume now that $g(\delta) = \xi$.
Then $(Q,f)$ contains a subquiver of the form
\[
  \xymatrix{
    2
    \ar@<.5ex>[r]^{\delta}
    & 4
    \ar@<.5ex>[l]^{\xi}
    \ar@(ru,dr)^{\varrho}[]
  } 
\]
with
$f(\delta) = \varrho$,
$f(\varrho) = \xi$,
$f(\xi) = \delta$,
and hence $g(\varrho) = \varrho$.
Then $\eta = \delta \xi$ and
\[
 \eta^2 
  = \delta \xi \delta \xi
  = \delta \xi f(\xi) \xi
  = c_{\varrho} \delta \varrho^{m_{\varrho}-1} \xi .
\]
We note that $m_{\varrho} = m_{\varrho} n_{\varrho} \geqslant 2$.
If $m_{\varrho} \geqslant 3$,
we have
$\delta \varrho^{m_{\varrho}-1} 
  = \delta  f(\delta) g(f(\delta)) \varrho^{m_{\varrho}-3} = 0$,
because $f^2(\delta) = \xi$ is not virtual.
Then again $\eta^2 = 0$, and we may take
$c_2 = 0$, $b_2 = 0$.
Finally, assume that $m_{\varrho} = 2$.
Then 
\[
 \eta^2 
  = c_{\varrho} \delta \varrho \xi
  = c_{\varrho} \delta  f(\delta) f^2(\delta)
  = c_{\varrho} c_{\delta} B_{\delta} 
  = c_{\varrho} c_{\delta} \delta \xi \gamma \mu \sigma \alpha \beta
  = c_{\varrho} c_{\delta} \eta \gamma \mu \sigma \alpha \beta
  = c_{\varrho} c_{\delta} \gamma \mu \sigma \alpha \beta \eta ,
\]
so we may take $c_2 = 0$ and $b_2 = c_{\varrho} c_{\delta}$.

Similarly, we prove that 
\[
 \mu^2 = c_3 \sigma \alpha \beta  \eta \gamma
               + b_3 \sigma \alpha \beta \eta \gamma \mu
\]
for some $c_3,b_3 \in K$.
\end{proof}

\section{Algebras $Q(2\cA)^k(b)$}
\label{sec:8}

In this section we classify the socle equivalences
of algebras of quaternion type $Q(2\cA)$
(introduced in \cite{E}).
These algebras will appear
in the proof of Theorem~\ref{th:main3}
(in Section~\ref{sec:10})
in the case of $\Lambda$ being
an algebra with two simple modules
not isomorphic to an algebra
$Q(2\cB)_3^t(b)$ with $t \geqslant 4$ and $b \in K^*$.

Let $(Q,f)$ be the triangulation quiver
\[
  \xymatrix{
    1
    \ar@(ld,ul)^{\alpha}[]
    \ar@<.5ex>[r]^{\beta}
    & 2
    \ar@<.5ex>[l]^{\gamma}
    \ar@(ru,dr)^{\eta}[]
  } 
\]
with $f$-orbits $(\alpha)$ and $(\beta\ \eta\ \gamma)$.
Then we have the $g$-orbits 
$\cO(\alpha) = (\alpha\ \beta\ \gamma)$
and
$\cO(\eta) = (\eta)$,
and
$\partial(Q,f) = \{ 1 \}$.
Let 
$m_{\bullet}:\cO(g)\to\bN^*$ 
be a weight function with
$m_{\cO(\alpha)} = k \geqslant 2$
and
$m_{\cO(\eta)} = 2$,
$c_{\bullet}:\cO(g)\to K^*$
a parameter function,
and
$b_{\bullet} : \partial(Q,f) \to K$
a border function.
We abbreviate 
$c_{\cO(\alpha)} = c$,
$c_{\cO(\eta)} = a$,
$d = b_1$.
Then the associated socle deformed weighted triangulation algebra
$\Lambda = \Lambda(Q,f,m_{\bullet},c_{\bullet},b_{\bullet})$
is given by the quiver $Q$ and the relations:
\begin{align*}
&&
  \alpha^2 &= c(\beta \gamma \alpha)^{k-1} \beta \gamma 
     + d(\beta \gamma \alpha)^k,
&
  \gamma \beta &= a \eta,
&
\alpha^2 \beta &= 0,
&
\beta \eta^2 &= 0,
&
\eta \gamma \alpha &= 0,
&&
\\
&&
\beta \eta &=  c(\alpha \beta \gamma)^{k-1} \alpha \beta,
&
\eta \gamma &= c(\gamma\alpha \beta )^{k-1} \gamma \alpha,
&
\gamma \alpha^2 &= 0,
&
\eta^2 \gamma &= 0,
&
\alpha \beta \eta &= 0.
&&
\end{align*}
We note that $\eta$ is a virtual loop of $(Q,f)$.

For a natural number $k \geqslant 2$ and $b \in K$,
we denote by $Q(2\cA)^k(b)$ the algebra given by 
the Gabriel quiver $Q_{\Lambda}$ of $\Lambda$
\[
  \xymatrix{
    1
    \ar@(ld,ul)^{\alpha}[]
    \ar@<.5ex>[r]^{\beta}
    & 2
    \ar@<.5ex>[l]^{\gamma}
  } 
\]
and the relations:
\begin{align*}
&&
  \alpha^2 &= (\beta \gamma \alpha)^{k-1} \beta \gamma 
     + b(\beta \gamma \alpha)^k,
&
\beta \gamma \beta &=  (\alpha \beta \gamma)^{k-1} \alpha \beta,
&
\gamma \beta \gamma &= (\gamma\alpha \beta )^{k-1} \gamma \alpha,
&
\alpha^2 \beta &= 0
&&
\end{align*}
(compare \cite[Theorem~VII.7.1]{E}).
We note that then
\begin{align*}
  \gamma \alpha^2 &= \gamma (\beta \gamma \alpha)^{k-1} \beta \gamma 
     + b \gamma(\beta \gamma \alpha)^k
  = \gamma \beta \gamma (\alpha \beta \gamma)^{k-1}
     + b \gamma \beta \gamma (\alpha \beta \gamma)^{k-1} \alpha
\\&
  = (\gamma\alpha \beta )^{k-1} \gamma \alpha^2 \beta \gamma (\alpha \beta \gamma)^{k-1}
    + b (\gamma\alpha \beta )^{k-1}  \gamma \alpha^2 \beta \gamma (\alpha \beta \gamma)^{k-1} \alpha 
     = 0 ,
\end{align*}
and 
\begin{align*}
  \alpha^3 &= (\beta \gamma \alpha)^{k-1} \beta \gamma \alpha
     + b (\beta \gamma \alpha)^k \alpha
  = (\beta \gamma \alpha)^{k}
     + b (\beta \gamma \alpha)^{k-1} \beta \gamma \alpha^{2}
  = (\beta \gamma \alpha)^{k}
,
\\
  \alpha^3 &= \alpha (\beta \gamma \alpha)^{k-1} \beta \gamma 
     + b \alpha(\beta \gamma \alpha)^k
  = (\alpha \beta \gamma)^{k}
     + b \alpha \beta (\gamma \alpha \beta)^{k-1} \gamma \alpha
\\&
  = (\alpha \beta \gamma)^{k}
     + b \alpha \beta \gamma \beta \gamma
  = (\alpha \beta \gamma)^{k}
     + b \alpha^2 \beta (\gamma \alpha \beta)^{k-1} \gamma 
  = (\alpha \beta \gamma)^{k}
.
\end{align*}

\begin{proposition}
\label{prop:8.1}
$\Lambda = \Lambda(Q,f,m_{\bullet},c_{\bullet},b_{\bullet})$
is isomorphic to $Q(2\cA)^k(b)$ for some $b \in K$.
Moreover, $b = 0$ if and only if $b_{\bullet} = 0$.
\end{proposition}

\begin{proof}
Since $\eta = a^{-1} \gamma \beta$,
we conclude that $\Lambda$ is given by the quiver
$Q_{\Lambda}$ and the relations:
\begin{align*}
&&
  \alpha^2 &= c (\beta \gamma \alpha)^{k-1} \beta \gamma 
     + d(\beta \gamma \alpha)^k,
&
\beta \gamma \beta &=  a c (\alpha \beta \gamma)^{k-1} \alpha \beta,
&
\alpha^2 \beta &= 0,
&
\beta \gamma \beta \gamma \beta &= 0,
&
\alpha \beta \gamma \beta &= 0,
&&
\\
&&&
&
\gamma \beta \gamma &= a c (\gamma\alpha \beta )^{k-1} \gamma \alpha,
&
\gamma \alpha^2 &= 0,
&
\gamma \beta \gamma \beta \gamma &= 0,
&
\gamma \beta \gamma \alpha &= 0.
&&
\end{align*}
We note that the relations 
$\gamma \alpha^2 = 0$, 
$\alpha \beta \gamma \beta = 0$, 
$\gamma \beta \gamma \alpha = 0$, 
$\beta \gamma \beta \gamma \beta = 0$ 
and
$\gamma \beta \gamma \beta \gamma = 0$ 
are superfluous,
because
\begin{align*}
\gamma \alpha^2 
  &= \gamma \big(c (\beta \gamma \alpha)^{k-1} \beta \gamma + d(\beta \gamma \alpha)^k\big)
   = c (\gamma \beta \gamma) (\alpha \beta \gamma)^{k-1}
     + d (\gamma \beta \gamma) \alpha  (\beta \gamma \alpha)^{k-1}
  \\&
   = a c^2 (\gamma\alpha \beta )^{k-1} \gamma (\alpha^2 \beta) \gamma (\alpha \beta \gamma)^{k-2}
     + a c d  (\gamma\alpha \beta )^{k-1} \gamma (\alpha^2 \beta) \gamma \alpha (\beta \gamma \alpha)^{k-2}
   = 0 ,
\\
  \alpha \beta \gamma \beta 
 &= a c \alpha (\alpha \beta \gamma)^{k-1} \alpha \beta
 = a c \alpha^2 \beta (\gamma \alpha \beta)^{k-1}
 = 0 ,
\\
  \gamma \beta \gamma \alpha 
 &= a c \big((\gamma\alpha \beta )^{k-1} \gamma \alpha\big) \alpha
 = a c (\gamma\alpha \beta )^{k-1} \gamma \alpha^2
 = 0 ,
\\
   \beta \gamma \beta \gamma \beta
 &= a c (\alpha \beta \gamma)^{k-1} \alpha \beta \gamma \beta
 = 0 ,
\\
   \gamma \beta \gamma \beta \gamma
 &=  a c \gamma \beta (\gamma\alpha \beta )^{k-1} \gamma \alpha
 =  a c (\gamma \beta \gamma\alpha) (\beta \gamma \alpha)^{k-1}
 = 0 .
\end{align*}
Let $u$ be an element in $K$ with $u^{3k-4} = a c$,
and
$\alpha^* = u \alpha$,
$\beta^* = u \beta$,
$\gamma^* = u \gamma$.
Then we have the relations
\begin{align*}
&&
\beta^* \gamma^* \beta^* &=  \big(\alpha^* \beta^* \gamma^*\big)^{k-1} \alpha^* \beta^*,
&
\gamma^* \beta^* \gamma^* &=  \big(\gamma^*\alpha^* \beta^*\big)^{k-1} \gamma^* \alpha^*,
&
\big(\alpha^*\big)^2 \beta^* &= 0
,
&&
\\
&&
&&
\!\!\!\!\!\!\!\!\!\!\!\!\!\!\!\!\!\!\!\!\!\!\!\!\!\!\!\!\!\!\!\!\!\!\!\!\!\!\!\!\!\!\!\!\!\!\!\!\!\!
\!\!\!\!\!\!\!\!\!\!\!\!\!\!\!\!\!\!\!\!\!\!\!\!\!\!\!\!\!\!\!\!\!\!\!\!\!\!\!\!\!\!\!\!\!\!\!\!\!\!
  \big(\alpha^*\big)^2 = c u^{-3(k-1)} \big(\beta^* \gamma^* \alpha^*\big)^{k-1} \beta^* \gamma^* 
     + d u^{-3k+2} \big(\beta^* \gamma^* \alpha^*\big)^k
.
\!\!\!\!\!\!\!\!\!\!\!\!\!\!\!\!\!\!\!\!\!\!\!\!\!\
\!\!\!\!\!\!\!\!\!\!\!\!\!\!\!\!\!\!\!\!\!\!\!\!\!\!\!\!\!\!\!\!\!\!\!\!\!
&
\end{align*}
Let $v = c u^{-3(k-1)}$ and
$w = d u^{-3k+2}$,
so we have the relation
\[
  \big(\alpha^*\big)^2 
  = v \big(\beta^* \gamma^* \alpha^*\big)^{k-1} \beta^* \gamma^* 
     + w \big(\beta^* \gamma^* \alpha^*\big)^k
     .
\]
Now we take $\lambda, \mu \in K$ satisfying
$\lambda^{3} v = \mu^4$ 
and  
$\lambda^k \mu^{2k-4} = 1$.
We set
$\alpha' = \lambda \alpha^*$,
$\beta' = \mu \beta^*$,
and
$\gamma' = \mu \gamma^*$.
Then we obtain the relations
\begin{align*}
&&
\beta' \gamma' \beta' &=  \big(\alpha' \beta' \gamma'\big)^{k-1} \alpha' \beta',
&
\gamma' \beta' \gamma' &=  \big(\gamma'\alpha' \beta'\big)^{k-1} \gamma' \alpha',
&
\big(\alpha'\big)^2 \beta' &= 0
,
&&
\\
&&
&&
\!\!\!\!\!\!\!\!\!\!\!\!\!\!\!\!\!\!\!\!\!\!\!\!\!\!\!\!\!\!\!\!\!\!\!\!\!\!\!\!\!\!\!\!\!\!\!\!\!\!
\!\!\!\!\!\!\!\!\!\!\!\!\!\!\!\!\!\!\!\!\!\!\!\!\!\!\!\!\!\!\!\!\!\!\!\!\!\!\!\!\!\!\!\!\!\!\!\!\!\!
  \big(\alpha'\big)^2 =  \big(\beta' \gamma' \alpha'\big)^{k-1} \beta' \gamma' 
     + b \big(\beta' \gamma' \alpha'\big)^k
\!\!\!\!\!\!\!\!\!\!\!\!\!\!\!\!\!\!\!\!\!\!\!\!\!\!\!\!\!\!\!\!\!\!\!
\!\!\!\!\!\!
&
\end{align*}
with 
$b = \lambda^2 \omega \mu^{-2k} \lambda^{-k} 
   = \lambda^2 \omega \mu^{-4}
   = \lambda^2 d u^{-3k+2} \mu^{-4}$.

Therefore, $\Lambda$ is isomorphic to $Q(2\cA)^k(b)$.
Clearly, $b = 0$ if and only if $d = 0$.
\end{proof}

\begin{proposition}
\label{prop:8.2}
Let $K$ be of characteristic $2$, 
$k \geqslant 2$ a natural number,
and $b,c\in K$.
Then the $K$-algebras 
$Q(2\cA)^k(b)$ and $Q(2\cA)^k(c)$
are isomorphic if and only if
$b^{5k-6} = c^{5k-6}$.
In particular, if $b \in K^*$, then
$Q(2\cA)^k(b)$ is not isomorphic to $Q(2\cA)^k(0)$.
\end{proposition}

\begin{proof}
(1)
Assume first that $b^{5k-6} = c^{5k-6}$ and $b,c\in K^*$.
We choose $d \in K$ with $d^2 = b^{-1} c$.
Then we have an isomorphism of algebras
$h : Q(2\cA)^k(b) \to Q(2\cA)^k(c)$
given by
\begin{align*}
&&
 h(\alpha) &= d^2 \alpha,
 &
 h(\beta) &= \beta,
 &
 h(\gamma) &= d^3 \gamma.
&&
\end{align*}
Indeed, we have the equalities in $Q(2\cA)^k(c)$
\begin{align*}
 h\big( &(\beta \gamma \alpha)^{k-1} \beta \gamma + b(\beta \gamma \alpha)^k \big)
 = \big( h(\beta) h(\gamma) h(\alpha) \big)^{k-1} h(\beta) h(\gamma)
    + b\big(h(\beta) h(\gamma) h(\alpha) \big)^k
 = d^{5k-2} (\beta \gamma \alpha)^{k-1} \beta \gamma
     + b d^{5k} (\beta \gamma \alpha)^k 
\\&
 = d^{5k-2} (\beta \gamma \alpha)^{k-1} \beta \gamma
     + d^{5k-2} c (\beta \gamma \alpha)^k 
 = d^{5k-2} \big( (\beta \gamma \alpha)^{k-1} \beta \gamma
       + c(\beta \gamma \alpha)^k \big)
 = d^{5k-2} \alpha^2
 = d^{5k-6} \big(d^{4} \alpha^2\big)
 = h(\alpha)^2,
\end{align*}
because 
$1 = (b^{-1}c)^{5k-6} = d^{2(5k-6)} =(d^{5k-6})^2$
implies $d^{5k-6} = 1$
($K$ is of characteristic $2$),
\begin{align*}
 h\big( (\alpha \beta \gamma)^{k-1} \alpha \beta \big)
&
 = \big( h(\alpha) h(\beta) h(\gamma) \big)^{k-1} h(\alpha) h(\beta)
 = d^{5k-3} (\alpha \beta \gamma)^{k-1} \alpha \beta
 = d^{5k-3} \beta \gamma \beta
 = d^{5k-6} \beta (d^3 \gamma) \beta
 =  h(\beta) h(\gamma)  h(\beta),
\\
 h\big( (\gamma \alpha \beta)^{k-1} \gamma \alpha \big)
&
 = \big( h(\gamma) h(\alpha) h(\beta) \big)^{k-1} h(\gamma) h(\alpha)
 = d^{5k} (\gamma \alpha \beta)^{k-1} \gamma \alpha
 = d^{5k} \gamma \beta \gamma
 = d^{5k-6} (d^3 \gamma) \beta (d^3 \gamma)
 =  h(\gamma) h(\beta) h(\gamma),
\end{align*}
and clearly 
$h(\alpha^2\beta) = h(\alpha)^2 h(\beta) = d^4 \alpha^2 \beta = 0$.

\smallskip

(2)
Assume that there exists an isomorphism of $K$-algebras
$\varphi : Q(2\cA)^k(b) \to Q(2\cA)^k(c)$.
We observe that $Q(2\cA)^k(c)$ has a $K$-basis given by the elements:
\begin{enumerate}[\quad(a)]
 \item
  $e_1$,
  $(\alpha \beta \gamma)^i \alpha$,
  $(\beta \gamma \alpha)^i \beta \gamma$,
  $(\alpha \beta \gamma)^j$,
  $(\beta \gamma \alpha)^j$,
  $(\alpha \beta \gamma)^k = \alpha^3 = (\beta \gamma \alpha)^k$,
  $0 \leqslant i \leqslant k-1$,
  $1 \leqslant j \leqslant k-1$,
 \item
  $\beta (\gamma \alpha \beta)^i$,
  $\alpha \beta (\gamma \alpha \beta)^i$,
  $0 \leqslant i \leqslant k-1$,
 \item
  $\gamma (\alpha \beta \gamma)^i$,
  $\gamma \alpha (\beta \gamma \alpha)^i$,
  $0 \leqslant i \leqslant k-1$,
 \item
  $e_2$,
  $\gamma \beta$,
  $(\gamma \alpha \beta)^j$,
  $1 \leqslant j \leqslant k$.
\end{enumerate}
Then $\varphi$ is given by the elements
\begin{align*}
\varphi(\alpha) &= A_0 + A_1 + A_2 + A_3
,\\
\varphi(\beta) &= B_0 + B_1
, \\
\varphi(\gamma) &= C_0 + C_1
,
\end{align*}
where we abbreviate
\begin{align*}
A_0 &= \sum_{i=0}^{k-1} a_{4i} \alpha (\beta \gamma \alpha)^i ,
&
A_1 &= \sum_{i=0}^{k-1} a_{4i+1} \beta \gamma (\alpha \beta \gamma)^i ,
&
A_2 &= \sum_{i=0}^{k-1} a_{4i+2} (\alpha \beta \gamma)^{i+1} ,
&
A_3 &= \sum_{i=0}^{k-2} a_{4i+3} (\beta \gamma \alpha)^{i+1} ,
\\
B_0 &= \sum_{i=0}^{k-1} b_{2i} \beta (\gamma \alpha \beta)^i ,
&
B_1 &=  \sum_{i=0}^{k-1} b_{2i+1} \alpha \beta (\gamma \alpha \beta)^i , 
&
C_0 &= \sum_{i=0}^{k-1} c_{2i} \gamma (\alpha \beta \gamma)^i ,
&
C_1 &= \sum_{i=0}^{k-1} c_{2i+1} \gamma \alpha (\beta \gamma \alpha)^i .
\end{align*}
We observe that 
\begin{align*}
A_0 A_0 &= a_0^2 \alpha^2, &
A_0 A_2 &= 0 = A_3 A_0, &
A_2 A_3 &= 0 = A_3 A_2, \\
A_1 A_1 &= 0, &
A_1 A_3 &= 0 = A_2 A_1, & 
C_0 B_0 &= b_0 c_0 \gamma \beta
,
\\
C_1 B_1 &= 0, &
C_0 B_0 C_1 &= 0, &
C_0 B_0 C_0 &= b_0 c_0^2 \gamma \beta \gamma , 
\\
B_1 C_0 B_0 &= 0 , & 
B_0 C_0 B_0 &= b_0^2 c_0 \beta \gamma \beta 
,
\end{align*}
because 
$\alpha^2 \beta = 0$,
$\alpha \beta \gamma \beta = 0$,
$\gamma \beta \gamma \alpha = 0$,
$\gamma \alpha^2 = 0$.

Further, we have the equalities
\begin{align*}
\varphi\big((\gamma \alpha \beta)^{k-1} \gamma \alpha\big)
 &= \big(\varphi(\gamma) \varphi(\alpha) \varphi(\beta)\big)^{k-1} 
                \varphi(\gamma) \varphi(\alpha)
 = a_0^k b_0^{k-1} c_0^k  (\gamma \alpha \beta)^{k-1} \gamma \alpha 
,\\
\varphi(\gamma \beta \gamma) &= 
\varphi(\gamma) \varphi(\beta) \varphi(\gamma) 
=
C_0 B_0 C_0 + C_0 B_1 C_0 + C_0 B_1 C_1 + C_1 B_0 C_0 + C_1 B_0 C_1
,
\end{align*}
with
\begin{align*}
C_0 B_1 C_0 &=
 \bigg(\sum_{i=0}^{k-1} c_{2i} \gamma (\alpha \beta \gamma)^i\bigg)
 \bigg(\sum_{j=0}^{k-1} b_{2j+1} \alpha \beta (\gamma \alpha \beta)^j \bigg)
 \bigg(\sum_{l=0}^{k-1} c_{2l} \gamma (\alpha \beta \gamma)^l\bigg)
\\&=
 \sum_{m=0}^{k-2} 
 \bigg(\sum_{i,j,l \geqslant 0; i+j+l = m}c_{2i} b_{2j+1}  c_{2l} \bigg)
 \gamma (\alpha \beta \gamma)^{m+1}
,
\\
C_0 B_1 C_1 &=
 \bigg(\sum_{i=0}^{k-1} c_{2i} \gamma (\alpha \beta \gamma)^i\bigg)
 \bigg(\sum_{j=0}^{k-1} b_{2j+1} \alpha \beta (\gamma \alpha \beta)^j \bigg)
 \bigg(\sum_{i=0}^{k-1} c_{2l+1} \gamma \alpha (\beta \gamma \alpha)^l\bigg)
\\&=
 \sum_{m=0}^{k-2} 
 \bigg(\sum_{i,j,l \geqslant 0; i+j+l = m}c_{2i} b_{2j+1}  c_{2l+1} \bigg)
 \gamma (\alpha \beta \gamma)^{m+1} \alpha
,
\\
C_1 B_0 C_0 &=
 \bigg(\sum_{i=0}^{k-1} c_{2i+1} \gamma \alpha (\beta \gamma \alpha)^i\bigg)
 \bigg(\sum_{j=0}^{k-1} b_{2j} \beta (\gamma \alpha \beta)^j \bigg)
 \bigg(\sum_{l=0}^{k-1} c_{2l} \gamma (\alpha \beta \gamma)^l\bigg)
\\&=
 \sum_{m=0}^{k-2} 
 \bigg(\sum_{i,j,l \geqslant 0; i+j+l = m}c_{2i+1} b_{2j}  c_{2l} \bigg)
 \gamma (\alpha \beta \gamma)^{m+1}
,
\end{align*}
\begin{align*}
C_1 B_0 C_1 &=
 \bigg(\sum_{i=0}^{k-1} c_{2i+1} \gamma \alpha (\beta \gamma \alpha)^i\bigg)
 \bigg(\sum_{j=0}^{k-1} b_{2j} \beta (\gamma \alpha \beta)^j \bigg)
 \bigg(\sum_{i=0}^{k-1} c_{2l+1} \gamma \alpha (\beta \gamma \alpha)^l\bigg)
\\&=
 \sum_{m=0}^{k-2} 
 \bigg(\sum_{i,j,l \geqslant 0; i+j+l = m}c_{2i+1} b_{2j}  c_{2l+1} \bigg)
 \gamma (\alpha \beta \gamma)^{m+1} \alpha
.
\end{align*}
Since
$\gamma \beta \gamma = (\gamma \alpha \beta)^{k-1} \gamma \alpha$
in 
$Q(2\cA)^k(b)$ and $Q(2\cA)^k(c)$,
we obtain the equalities 
\begin{align}
\label{eq:sumBC}
\sum_{i,j,l \geqslant 0; i+j+l = m}c_{2i} b_{2j+1}  c_{2l} 
 + \sum_{i,j,l \geqslant 0; i+j+l = m}c_{2i+1} b_{2j}  c_{2l}
&= 0,
\\
\notag
\sum_{i,j,l \geqslant 0; i+j+l = n}c_{2i} b_{2j+1}  c_{2l+1}
 + \sum_{i,j,l \geqslant 0; i+j+l = n}c_{2i+1} b_{2j}  c_{2l+1}
&= 0,
\end{align}
for $m = 0, \dots, k-2$,  $n = 0, \dots, k-3$,
and
\begin{align}
\label{eq:rel-cbc}
\sum_{i,j,l \geqslant 0; i+j+l = k-2}\!\!\!\!\!\!\!\!c_{2i} b_{2j+1}  c_{2l+1}
 + \sum_{i,j,l \geqslant 0; i+j+l = k-2}\!\!\!\!\!\!\!\!c_{2i+1} b_{2j}  c_{2l+1}
 + a_0^k b_0^{k-1} c_0^k
&= b_0 c_0^2
.
\end{align}
Similarly, we have the equalities
\begin{align*}
\varphi\big( (\beta \gamma \alpha)^{k-1} \beta \gamma \big)
 &= \big(\varphi(\beta) \varphi(\gamma) \varphi(\alpha)\big)^{k-1} 
                \varphi(\beta) \varphi(\gamma) 
 = a_0^{k-1} (b_0 c_0)^k (\beta \gamma \alpha)^{k-1} \beta \gamma
 + (a_0 b_0 c_0)^{k-1} (b_1c_0 + b_0 c_1) (\beta \gamma \alpha)^k
,\\
\varphi\big(b (\beta \gamma \alpha)^k\big)
 &= b \big(\varphi(\beta) \varphi(\gamma) \varphi(\alpha)\big)^{k} 
 = b (a_0 b_0 c_0)^k (\beta \gamma \alpha)^k
,\\
\varphi\big(\alpha^2\big)
 &= \varphi(\alpha)^2 =
  \sum_{0\leqslant i,j \leqslant 3} A_i A_j
,
\end{align*}
where
\begin{align*}
A_0 A_1 &= 
 \bigg(\sum_{i=0}^{k-1} a_{4i} \alpha (\beta \gamma \alpha)^i \bigg)
 \bigg(\sum_{j=0}^{k-1} a_{4j+1} \beta \gamma (\alpha \beta \gamma)^j \bigg)
=
 \sum_{m=0}^{k-1} 
 \bigg(\sum_{i,j \geqslant 0; i+j = m} a_{4i} a_{4j+1} \bigg)
 (\alpha \beta \gamma)^{m+1}
,\\
A_2 A_2 &= 
 \bigg( \sum_{i=0}^{k-1} a_{4i+2} (\alpha \beta \gamma)^{i+1} \bigg)^2
=
 \sum_{m=0}^{k-2} 
 \bigg(\sum_{i,j \geqslant 0; i+j = m} a_{4i+2} a_{4j+2} \bigg)
 (\alpha \beta \gamma)^{m+2}
,
\\
A_1 A_0 &= 
 \bigg(\sum_{i=0}^{k-1} a_{4i+1} \beta \gamma (\alpha \beta \gamma)^i \bigg)
 \bigg(\sum_{j=0}^{k-1} a_{4j} \alpha (\beta \gamma \alpha)^j \bigg)
=
 \sum_{m=0}^{k-1} 
 \bigg(\sum_{i,j \geqslant 0; i+j = m} a_{4i+1} a_{4j} \bigg)
 (\beta \gamma \alpha)^{m+1}
,\\
A_3 A_3 &= 
 \bigg( \sum_{i=0}^{k-2} a_{4i+3} (\beta \gamma \alpha)^{i+1} \bigg)^2
 =
 \sum_{m=0}^{k-2} 
 \bigg(\sum_{i,j \geqslant 0; i+j = m} a_{4i+3} a_{4j+3} \bigg)
 (\beta \gamma \alpha)^{m+2}
,\\
A_0 A_3 &= 
 \bigg(\sum_{i=0}^{k-1} a_{4i} \alpha (\beta \gamma \alpha)^i \bigg)
 \bigg(\sum_{j=0}^{k-1} a_{4j+3} (\beta \gamma \alpha)^{j+1} \bigg)
=
 \sum_{m=0}^{k-2} 
 \bigg(\sum_{i,j \geqslant 0; i+j = m} a_{4i} a_{4j+3} \bigg)
 (\alpha \beta \gamma)^{m+1} \alpha
,\\
A_2 A_0 &= 
 \bigg( \sum_{i=0}^{k-1} a_{4i+2} (\alpha \beta \gamma)^{i+1} \bigg)
 \bigg(\sum_{j=0}^{k-1} a_{4j} \alpha (\beta \gamma \alpha)^j \bigg)
 =
 \sum_{m=0}^{k-2} 
 \bigg(\sum_{i,j \geqslant 0; i+j = m} a_{4i+2} a_{4j} \bigg)
 (\alpha \beta \gamma)^{m+1} \alpha
,\\
A_1 A_2 &= 
 \bigg(\sum_{i=0}^{k-1} a_{4i+1} \beta \gamma (\alpha \beta \gamma)^i \bigg)
 \bigg(\sum_{j=0}^{k-1} a_{4j+2}  (\alpha \beta \gamma)^{j+1} \bigg)
=
 \sum_{m=0}^{k-2} 
 \bigg(\sum_{i,j \geqslant 0; i+j = m} a_{4i+1} a_{4j+2} \bigg)
\beta \gamma (\alpha \beta \gamma)^{m+1}
,
\\
A_3 A_1 &= 
 \bigg( \sum_{i=0}^{k-2} a_{4i+3} (\beta \gamma \alpha)^{i+1} \bigg)
 \bigg(\sum_{j=0}^{k-1} a_{4j+1} \beta \gamma (\alpha \beta \gamma)^j \bigg)
 =
 \sum_{m=0}^{k-2} 
 \bigg(\sum_{i,j \geqslant 0; i+j = m} a_{4i+3} a_{4j+1} \bigg)
 \beta \gamma (\alpha \beta \gamma)^{m+1}
.
\end{align*}

Since,
$\alpha^2 = (\beta \gamma \alpha)^{k-1} \beta \gamma 
     + b (\beta \gamma \alpha)^k$
in 
$Q(2\cA)^k(b)$ 
and 
$\alpha^2 = (\beta \gamma \alpha)^{k-1} \beta \gamma 
     + c (\beta \gamma \alpha)^k$
in 
$Q(2\cA)^k(c)$,
we 
have
\begin{align}
\label{eq:rel-aa}
a_0^{k-1} (b_0 c_0)^k
- \sum_{i,j \geqslant 0; i+j = k-2} a_{4i+1} a_{4j+2}
- \sum_{i,j \geqslant 0; i+j = k-2} a_{4i+3} a_{4j+1}
 &= a_0^2
,\\
\label{eq:rel-aa-soc}
  (a_0 b_0 c_0)^{k-1} (b_1c_0 + b_0 c_1) 
 + b (a_0 b_0 c_0)^k
-\sum_{i,j \geqslant 0; i+j = k-1} a_{4i} a_{4j+1} 
\\\notag
- \sum_{i,j \geqslant 0; i+j = k-2} a_{4i+2} a_{4j+2} 
- \sum_{i,j \geqslant 0; i+j = k-1} a_{4i+1} a_{4j} 
- \sum_{i,j \geqslant 0; i+j = k-2} a_{4i+3} a_{4j+3}
 &= c a_0^2
,
\\
\notag
 a_0 a_1 &= 0
,\\
\notag
\sum_{i,j \geqslant 0; i+j = n+1} a_{4i} a_{4j+1} 
+ \sum_{i,j \geqslant 0; i+j = n} a_{4i+2} a_{4j+2}
&= 0
,
\end{align}
\begin{align}
\notag
\sum_{i,j \geqslant 0; i+j = n+1} a_{4i+1} a_{4j} 
 +\sum_{i,j \geqslant 0; i+j = n} a_{4i+3} a_{4j+3}
 &= 0
,
\\
\label{eq:sumA}
\sum_{i,j \geqslant 0; i+j = m} a_{4i} a_{4j+3}
+ \sum_{i,j \geqslant 0; i+j = m} a_{4i+2} a_{4j}
&= 0
,\\
\notag
\sum_{i,j \geqslant 0; i+j = m} a_{4i+1} a_{4j+2}
+ \sum_{i,j \geqslant 0; i+j = m} a_{4i+3} a_{4j+1}
&= 0
,
\end{align}
for $m = 0, \dots, k-2$,  $n = 0, \dots, k-3$.

We prove now inductively that
\begin{align}
\label{eq:a23}
 a_{4i+2} + a_{4i+3} &= 0, \mbox{ for } i = 0, \dots, k-2 .
\end{align}
Indeed, recall that $a_0 \neq 0$ from assumption.
Then from \eqref{eq:sumA} with $m = 0$ we have
$a_0 a_2 + a_0 a_3 = 0$, and hence $a_2 + a_3  = 0$.
Let $m \in \{1,\dots, k-2\}$ and 
assume that  $a_{4i+2} + a_{4i+3} = 0$ for $i < m$.
Then, applying again \eqref{eq:sumA}, we obtain
\begin{align*}
a_0(a_{4m+2} + a_{4m+3}) 
&= 
\sum_{i=0}^m 
a_{4(m-i)}(a_{4i+2} + a_{4i+3}) 
=
\sum_{i,j \geqslant 0; i+j = m} a_{4i} a_{4j+3}
+ \sum_{i,j \geqslant 0; i+j = m} a_{4i+2} a_{4j}
=0
.
\end{align*}
Then we conclude that $a_{4m+2} + a_{4m+3} = 0$ 
(because $a_0 \neq 0$), so it proves \eqref{eq:a23}.

Dually we show the equalities
\begin{align}
\label{eq:bc}
\sum_{i,j \geqslant 0; i+j = m}c_{2i} b_{2j+1}  
 + \sum_{i,j \geqslant 0; i+j = m}c_{2i+1} b_{2j} 
&= 0,
\mbox{ for } m = 0, \dots, k-2 .
\end{align}
We obtain from \eqref{eq:sumBC} with $m = 0$ the equality
\[
(c_{0} b_{1} + c_1 b_0) c_0 = 0 ,
\]
so, since $c_0 \neq 0$, we obtain the required equality
\[
c_{0} b_{1} + c_1 b_0 = 0 . 
\]
Let $m \in \{1,\dots, k-2\}$ and 
assume that  
$\sum_{i,j \geqslant 0; i+j = l}c_{2i} b_{2j+1}  
+ \sum_{i,j \geqslant 0; i+j = l}c_{2i+1} b_{2j} 
= 0$
for $l < m$.
Then, applying it to \eqref{eq:sumBC}, we obtain
\begin{align*}
\bigg(\sum_{i,j \geqslant 0; i+j = m}c_{2i} b_{2j+1}  
 + \sum_{i,j \geqslant 0; i+j = m}c_{2i+1} b_{2j} \bigg) c_0
&= 
\sum_{l=0}^{m}
\bigg(\sum_{i,j \geqslant 0; i+j = m-l}c_{2i} b_{2j+1}  
 + \sum_{i,j \geqslant 0; i+j = m-l}c_{2i+1} b_{2j} \bigg) c_{2l}
\\&= 
\sum_{i,j,l \geqslant 0; i+j+l = m}c_{2i} b_{2j+1}  c_{2l} 
 + \sum_{i,j,l \geqslant 0; i+j+l = m}c_{2i+1} b_{2j}  c_{2l}
=0
.
\end{align*}
Therefore, we obtain \eqref{eq:bc} because $c_0 \neq 0$.

Further, applying \eqref{eq:bc} to \eqref{eq:rel-cbc}
we obtain
\begin{align}
\label{eq:rel-cbc-}
a_0^k b_0^{k-1} c_0^k
&= b_0 c_0^2
.
\end{align}
Moreover, applying \eqref{eq:a23} to \eqref{eq:rel-aa}
we obtain
\begin{align}
\label{eq:rel-aa-}
a_0^{k-1} (b_0 c_0)^k
 &= a_0^2
 .
\end{align}
Furthermore, applying \eqref{eq:a23} and \eqref{eq:bc} to \eqref{eq:rel-aa-soc}
we obtain
\begin{align}
\label{eq:rel-aa--}
b (a_0 b_0 c_0)^k 
=
b (a_0 b_0 c_0)^k
- 2 \sum_{i,j \geqslant 0; i+j = k-1} a_{4i} a_{4j+1} 
 &= c a_0^2 .
\end{align}
We note that \eqref{eq:rel-cbc-} is equivalent to
\begin{align}
\label{eq:rel-unity}
a_0^{k} (b_0 c_0)^{k-2}
 &= 1_K ,
\end{align}
and
from \eqref{eq:rel-cbc-} and \eqref{eq:rel-aa-} it follows that
\begin{gather}
\label{eq:rel-a-bc}
 a_0^3 =  (a_0 b_0 c_0)^k = (b_0 c_0)^2 .
\end{gather}
Applying \eqref{eq:rel-a-bc} to \eqref{eq:rel-unity}
we obtain
\begin{align}
\label{eq:rel-unity-}
a_0^{5k-6} 
 &= a_0^{2k} \big(a_0^3\big)^{k-2}
 = \Big(a_0^{k} (b_0 c_0)^{k-2}\Big)^2
 = 1_K 
 .
\end{align}
Further,
applying \eqref{eq:rel-a-bc} to \eqref{eq:rel-aa--} 
we obtain
\begin{align}
\label{eq:rel-aa---}
c &=
  b a_0  
  - 2 a_0^{-2} \sum_{i,j \geqslant 0; i+j = k-1} a_{4i} a_{4j+1} 
 = b a_0
 .
\end{align}
Finally,
we conclude 
from \eqref{eq:rel-unity-} and \eqref{eq:rel-aa---}
that
$b^{5k-6} = (b a_0)^{5k-6} = c^{5k-6}$.
\end{proof}

\section{Algebras $Q(2\cB)_3^t(a,b)$}
\label{sec:9}

We classify now the socle equivalences
of algebras of quaternion type $Q(2\cB)_3$
(introduced in \cite{E}).
Results of this section will be applied
in the proof of Theorem~\ref{th:main3}
in Section~\ref{sec:10}
in the case of $\Lambda$ being
an algebra with two simple modules.

Let $(Q,f)$ be the triangulation quiver
\[
  \xymatrix{
    1
    \ar@(ld,ul)^{\alpha}[]
    \ar@<.5ex>[r]^{\beta}
    & 2
    \ar@<.5ex>[l]^{\gamma}
    \ar@(ru,dr)^{\eta}[]
  } 
\]
with $f$-orbits $(\alpha)$ and $(\beta\ \eta\ \gamma)$.
Then we have the $g$-orbits 
$\cO(\alpha) = (\alpha\ \beta\ \gamma)$
and
$\cO(\eta) = (\eta)$,
and
$\partial(Q,f) = \{ 1 \}$.

For a natural number $t \geqslant 3$, $a \in K^*$, $b \in K$,
we denote by $Q(2\cB)_3^t(a,b)$ the algebra given by 
the quiver $Q$
and the relations:
\begin{align*}
&&
\alpha \beta &= \beta \eta, &
\eta \gamma &= \gamma \alpha, &
 \alpha^2 &= \beta \gamma + b \alpha^3, &
\gamma \beta &= a \eta^{t-1} , 
&&
\\
&&
\alpha^4 &= 0, &
\eta^{t+1} &= 0 , &
\gamma \alpha^2 &= 0, &
\alpha^2 \beta &= 0 
&&
\end{align*}
(compare \cite[Theorem~VII.7.3(ii)]{E}).

\begin{proposition}
\label{prop:9.1}
Let $t \geqslant 3$ be a natural number and 
$\Lambda = \Lambda(Q,f,m_{\bullet},c_{\bullet},b_{\bullet})$
be a socle deformed weighted triangulation algebra
of the above triangulation quiver $(Q,f)$
with 
$m_{\cO(\alpha)} = 1$ 
and
$m_{\cO(\eta)} = t$. 
Then $\Lambda$ is isomorphic to an algebra
$Q(2\cB)_3^t(a,b)$
with
$a \in K \setminus \{ 0,1 \}$
for 
$t = 3$, and $a = 1$ for $t \geqslant 4$.
Moreover, $b = 0$ if and only if $b_{\bullet}$ is zero.
\end{proposition}

\begin{proof}
Let
$c_1 = c_{\cO(\alpha)}$,
$c_2 = c_{\cO(\eta)}$,
and
$d = b_1$.
Since 
$m_{\cO(\alpha)} = 1$, 
it follows from the general assumption that
$\eta$ is not virtual, and hence
$t = m_{\cO(\eta)} \geqslant 3$
(see also \cite[Example~3.1(2)]{ES5}.
The algebra $\Lambda$ is given by the quiver $Q$
and the relations: 
\begin{align*}
&&
\beta \eta &= c_1 \alpha \beta , &
\eta \gamma &= c_1 \gamma \alpha, &
\gamma \beta &= c_2 \eta^{t-1} , &
\alpha^2 &= c_1 \beta \gamma + b_1 \beta \gamma \alpha, 
&&
\\
&&
\alpha^2 \beta &= 0, &
\beta \eta^2 &= 0, &
\eta \gamma \alpha &= 0, &
\gamma \beta \gamma &= 0, 
&&\\&&
\alpha \beta \eta &= 0, &
\beta \gamma \beta &= 0, &
\eta^2 \gamma &= 0, &
\gamma \alpha^2 &= 0. 
&&
\end{align*}
Moreover, we have 
\begin{align*}
 \alpha^3 &= c_1\beta \gamma \alpha + b_1 \beta \gamma \alpha^2
   =  c_1\beta \gamma \alpha , 
\qquad\mbox{and hence}
\\
 \alpha^2 &= c_1 \beta \gamma + c_1^{-1} b_1 \alpha^3. 
\end{align*}
A direct checking shows that $\Lambda$ is given by the quiver $Q$,
the four commutativity relations, and the four zero relations
\begin{align*}
&&
\alpha^4 &= 0, &
\eta^{t+1} &= 0 , &
\alpha^2 \beta &= 0, &
\gamma \alpha^2 &= 0. 
&&
\end{align*}
Let $u$ be an element in $K$ with 
$u^2 = c_1^{-1}$,
and
$\alpha^* = u \alpha$,
$\eta^* = u^3 \eta$.
Then we have the equalities
\begin{align*}
&&
\alpha^* \beta &= \beta \eta^* , &
\eta^* \gamma &= \gamma \alpha^*, &
\gamma \beta &= c_2 u^{-3(t-1)} (\eta^*)^{t-1} , &
\big(\alpha^*\big)^2 &=  \beta \gamma + u^4 b_1 \alpha^3, 
&&
\\
&&
\big(\alpha^*\big)^4 &= 0, &
\big(\eta^*\big)^{t+1} &= 0 , &
\big(\alpha^*\big)^2 \beta &= 0, & 
\gamma \big(\alpha^*\big)^2 &= 0. 
&&
\end{align*}
Therefore, $\Lambda$ is isomorphic to the algebra
$Q(2\cB)_3^t(a,b)$
with
$a = c_2 u^{-3(t-1)}\in K^*$
and
$b = u^4 b_1\in K$.
We note that $b = 0$ if and only if $b_1 = 0$.

Assume now that $t \geqslant 4$.
We take $v \in K$ such that
$v^{t-3} = a$,
and 
$\alpha' = v \alpha^*$,
$\beta' = v \beta$,
$\gamma' = v \gamma$,
$\eta' = v \eta^*$.
Then we obtain the equalities
\begin{align*}
&&
\alpha' \beta' &= \beta' \eta', &
\eta' \gamma' &= \gamma' \alpha', &
\gamma' \beta' &= (\eta')^{t-1} , &
 (\alpha')^2 &= \beta' \gamma' + v^{-1} b (\alpha')^3, 
&&
\\
&&
(\alpha')^4 &= 0, &
(\eta')^{t+1} &= 0 , &
(\alpha')^2 \beta' &= 0, &
\gamma' (\alpha')^2 &= 0. 
&&
\end{align*}
Therefore, $\Lambda$ is isomorphic to 
$Q(2\cB)_3^t(1,b')$
with $b' = v^{-1} b$.

Finally, assume that $t = 3$.
Then $\Lambda' = Q(2\cB)_3^t(a,0)$
is the disc algebra $D(a)$ considered in
\cite[Example~3.1(1)]{ES5},
and consequently a non-singular disc algebra,
by our general assumption.
Hence $a \in K \setminus \{0,1\}$.
\end{proof}

For $t \geqslant 4$ and $b \in K$, we abbreviate
$Q(2\cB)_3^t(b) = Q(2\cB)_3^t(1,b)$.
If $K$ is of characteristic different from $2$,
then it follows from Proposition~\ref{prop:2.3}
that $Q(2\cB)_3^t(b)$ is isomorphic to
$Q(2\cB)_3^t(0)$.
Otherwise, we have the following surprising fact.

\begin{proposition}
\label{prop:9.2}
Let $K$ be of characteristic $2$, 
and $t \geqslant 4$.
\begin{enumerate}[\rm (i)]
 \item
  If $t$ is even, then 
  $Q(2\cB)_3^t(b)$ is isomorphic to $Q(2\cB)_3^t(0)$
  for any $b \in K$.
 \item
  If $t$ is odd and $b,c\in K$, then 
  the algebras $Q(2\cB)_3^t(b)$ and $Q(2\cB)_3^t(c)$
  are isomorphic if and only if $b^{t-3} = c^{t-3}$.
\end{enumerate}
\end{proposition}

\begin{proof}
(i)
Assume $t$ is even and $b \in K$.
Then there is an isomorphism of $K$-algebras
 $h : Q(2\cB)_3^t(b) \to Q(2\cB)_3^t(0)$
given by
$h(\alpha) = \alpha$,
$h(\beta) = \beta$,
$h(\gamma) = \gamma + b \gamma \alpha$,
$h(\eta) = \eta + b \eta^2$.
Indeed, we have the equalities in $Q(2\cB)_3(0)$
\begin{align*}
 h(\alpha \beta) 
&= h(\alpha) h(\beta) 
 = \alpha \beta
 = \beta \eta
 = \beta\big(\eta + b \eta^2\big)
 = h(\beta) h(\eta) 
 = h(\beta \eta) ,
\\
 h(\eta \gamma) 
&= h(\eta) h(\gamma) 
 = \big(\eta + b \eta^2\big) (\gamma + b \gamma \alpha) 
 = \eta \gamma
  = \gamma \alpha
  = (\gamma + b \gamma \alpha)  \alpha
 = h(\gamma) h(\alpha) 
 = h(\gamma \alpha) ,
 \\
 h(\beta \gamma) 
&=h(\beta)  h(\gamma) 
 = \beta (\gamma + b \gamma \alpha) 
 = \beta\gamma+ b \beta  \gamma \alpha 
 = \alpha^2 + b \alpha^3
 = h(\alpha)^2 + b h(\alpha)^3
 = h\big(\alpha^2 + b \alpha^3\big),
 \\
 h(\gamma \beta) 
&= h(\gamma) h(\beta) 
 = (\gamma + b \gamma \alpha)  \beta
 = \gamma \beta+ b \gamma \alpha \beta
 = \eta^{t-1} + b \eta^{t}
 = \eta^{t-1} + (t-1)b \eta^{t}
 = \big(\eta + b \eta^2\big)^{t-1}
 = h(\eta)^{t-1} 
 = h\big(\eta^{t-1}\big)  ,
\end{align*}
because 
$\gamma \alpha \beta = \gamma \beta \eta = \eta^{t-1} \eta = \eta^t$,
$t$ is even, and $K$ is of characteristic $2$.

\smallskip

(ii)
Assume $t$ is odd and $b,c \in K$.

\smallskip

\quad
(1)
Assume that $b^{t-3} = c^{t-3}$.
Then there is an isomorphism of $K$-algebras
$\psi : Q(2\cB)_3^t(b) \to Q(2\cB)_3^t(c)$
 given by 
$\psi(\alpha) = b^{-1} c \alpha$,
$\psi(\beta) = b^{-1} c \beta$,
$\psi(\gamma) = b^{-1} c \gamma$,
$\psi(\eta) = b^{-1} c \eta$.

Observe that the following equalities hold in $Q(2\cB)_3^t(c)$
\begin{align*}
 \psi(\beta \gamma) 
&=\psi(\beta)  \psi(\gamma) 
 = b^{-2} c^2 \beta \gamma
 = b^{-2} c^2 \alpha^2 + b^{-2} c^3  \alpha^3
 = \big(b^{-1} c \alpha\big)^2 + b \big(b^{-1} c \alpha\big)^3
 = \psi(\alpha)^2 + b \psi(\alpha)^3
 = \psi\big(\alpha^2 + b \alpha^3\big),
 \\
 \psi(\gamma \beta) 
&= \psi(\gamma) \psi(\beta) 
 = b^{-2} c^2 \gamma \beta
 = b^{-2} c^2 \eta^{t-1}
 = \big(b^{-1} c \eta\big)^{t-1}
 = \psi(\eta)^{t-1} 
 = \psi\big(\eta^{t-1}\big)  ,
\end{align*}
because 
$\big(b^{-1} c \big)^{t-1} = b^{-2} c^2  b^{-(t-3)} c^{t-3} = b^{-2} c^2$.

\smallskip

\quad
(2)
Assume that there exists a $K$-algebra isomorphism 
$\varphi : Q(2\cB)_3^t(b) \to Q(2\cB)_3^t(c)$.
Then $\varphi$ is given by 
\begin{align*}
&&
 \varphi(\alpha) &= a_0 \alpha + a_1 \alpha^2  + a_2 \alpha^3 ,
 &
 \varphi(\beta) &= b_0 \beta + b_1 \alpha \beta ,
 &
 \varphi(\gamma) &= c_0 \gamma + c_1 \gamma \alpha ,
 &
 \varphi(\eta) &= \sum_{i=0}^{t-1} d_i \eta^{i+1} ,
&&
\end{align*}
for some 
$a_0,b_0,c_0,d_0 \in K^*$, 
$a_1,a_2,b_1,c_1,d_i \in K$, for 
$1 \leqslant i \leqslant t-1$.
We have the equalities in $Q(2\cB)_3^t(c)$
\begin{align*}
0 &= \varphi(0) = \varphi (\alpha \beta - \beta \eta) 
  =  \varphi (\alpha)  \varphi (\beta) - \varphi (\beta) \varphi (\eta)
 = 
  a_0 b_0 \alpha \beta -  b_0 d_0 \beta \eta
 =
  (a_0 - d_0) b_0 \alpha \beta ,
\\
0 &= \varphi(0) = \varphi \big(\beta \gamma - \alpha^2 - b \alpha^3\big) 
  =  \varphi (\beta)  \varphi (\gamma) - \varphi (\alpha)^2 - b \varphi (\alpha)^3
 \\&
 =
   b_0 c_0 \beta \gamma
  - a_0^2 \alpha^2
  - \big(2 a_0 a_1 + b a_0^3 - (b_0 c_1 + b_1 c_0) \big) \alpha^3
 =
   (b_0 c_0 - a_0^2) \alpha^2
  + \big(b_0 c_0 c + (b_0 c_1 + b_1 c_0) - b a_0^3 \big) \alpha^3
,
\\
0 &= \varphi(0) = \varphi \big(\gamma \beta - \eta^{t-1}\big) 
  =  \varphi (\gamma) \varphi (\beta) 
   - \varphi (\eta)^{t-1}
 =
   b_0 c_0 \gamma \beta
  - d_0^{t-1} \eta^{t-1}
  - \big((t-1)d_0^{t-2} d_1 - (b_0 c_1 + b_1 c_0)\big) \eta^{t}
 \\&
 =
   \big(b_0 c_0 - d_0^{t-1}\big) \eta^{t-1}
  + (b_0 c_1 + b_1 c_0) \eta^{t}
,
\end{align*}
because 
$t$ is odd and $K$ is of characteristic $2$.
Hence we conclude that
\begin{align*}
&&
a_0 &= d_0,
&
a_0^2 &= b_0 c_0 = d_0^{t-1},
&
b_0 c_1 + b_1 c_0 &= 0,
&
b_0 c_0 c &= b a_0^3.
&&
\end{align*}
In particular, we have  
$d_0^{t-3} = d_0^{-2} b_0 c_0 = a_0^{-2} b_0 c_0 = 1$
and
$b^{-1} c = (b_0 c_0)^{-1} a_0^3 = a_0^{-2} a_0^{3} = a_0 = d_0$.
Therefore, $b^{t-3} = c^{t-3}$.
\end{proof}

The following proposition follows from the proof of 
\cite[Lemma~5.4]{BS}.

\begin{proposition}
\label{prop:9.3}
Let $K$ be of characteristic $2$, 
and $a \in K \setminus \{0,1\}$.
Then the following statements hold.
\begin{enumerate}[\rm (i)]
 \item
  For any $b \in K^*$,
  the algebras $Q(2\cB)_3^3(a,b)$ and $Q(2\cB)_3^3(a,1)$
  are isomorphic. 
 \item
  The algebras $Q(2\cB)_3^3(a,1)$ and $Q(2\cB)_3^3(a,0)$
  are not isomorphic.
\end{enumerate}
\end{proposition}

\section{Proof of Theorem~\ref{th:main3}}
\label{sec:10}

We present first the following known general result
(see \cite[Theorem~2.9]{Ku})
with its elementary proof, for convenience of the reader.

\begin{proposition}
\label{prop:10.1}
Let $A$ and $B$ be isomorphic basic algebras,
and $e_1,\dots,e_n$ and $f_1,\dots,f_n$
complete sets of orthogonal primitive idempotents  
of $A$ and $B$, respectively.
Then there exist a $K$-algebra isomorphism $\varphi : A \to B$
and a permutation $\sigma$ of $\{1,\dots,n\}$ such that 
$\varphi(e_i) = f_{\sigma(i)}$ for any $i \in \{1,\dots,n\}$.
\end{proposition}

\begin{proof}
Let $\psi : A \to B$ be a $K$-algebra isomorphism.
Since $\psi(e_1), \dots, \psi(e_n)$ is a complete set of
pairwise orthogonal primitive idempotents  of $B$,
we have two decompositions
\[
  B = \bigoplus_{i=1}^n \psi(e_i) B = \bigoplus_{i=1}^n f_i B
\]
into direct sum of indecomposable projective $B$-modules.
By the Krull-Schmidt Theorem there exists a permutation
$\sigma$ of $\{1,\dots,n\}$ such that 
$\psi(e_i) B = f_{\sigma(i)} B$ in $\mod B$
for any  $i \in \{1,\dots,n\}$.
The isomorphisms are induced by left multiplications
of elements of $B$ by elements 
$u_i \in f_{\sigma(i)} B \psi(e_i)$
and
$v_i \in \psi(e_i) B f_{\sigma(i)}$
satisfying 
$u_i v_i = f_{\sigma(i)}$
and
$v_i u_i = \psi(e_i)$
for any  $i \in \{1,\dots,n\}$.
Then $u = \sum_{i=1}^n u_i$ is invertible 
to $v = \sum_{i=1}^n v_i$ in $B$.
We can define the inner automorphism $\theta$ of $B$
by $\theta(B) = u b u^{-1}$ for all $b \in B$, 
and the composition $\varphi = \theta \psi$.
Then $\varphi : A \to B$ is a $K$-algebra isomorphism
such that $\varphi(e_i) = f_{\sigma(i)}$
for any  $i \in \{1,\dots,n\}$.
\end{proof}

Assume that
$\Lambda = \Lambda(S,\vec{T},m_{\bullet},c_{\bullet},b_{\bullet})$
is a socle deformed weighted surface algebra,
with a non-zero border function $b_{\bullet}$
over an algebraically closed field $K$ of characteristic $2$,
and $\bar{\Lambda}=\Lambda(S, \vec{T}, m_{\bullet}, c_{\bullet})$.
We recall that, by our general  assumption, 
$\bar{\Lambda}$ is not the singular disc algebra $D(1)$.

\smallskip

Let $(Q,f) = (Q(S,\vec{T}),f)$, and $Q = (Q_0,Q_1,s,t)$.
Since the border function $b_{\bullet}$ is non-zero,
there is a border loop $\alpha$ of $(Q,f)$ with 
$b_{s(\alpha)} \neq 0$.
Moreover, we have $\bar{\alpha} = g(\alpha)$
and $n_{\alpha} \geqslant 3$.

\smallskip

(1)
Assume that $|Q_0|\geqslant3$.
We note that then $f(\bar{\alpha})$ is not a loop.
Then the claim follows by
\begin{itemize}
 \item
  Propositions \ref{prop:5.1} and \ref{prop:5.2}, 
  if $m_{\cO(\alpha)} \geqslant 2$;
 \item
  Lemma \ref{lem:6.1} and Proposition \ref{prop:6.2}, 
  if $m_{\cO(\alpha)} = 1$ 
  and $\cO(\alpha) \neq \cO(f(\bar{\alpha}))$;
 \item
  Lemma \ref{lem:7.1} and Proposition \ref{prop:7.2}, 
  if $m_{\cO(\alpha)} = 1$ 
  and $\cO(\alpha) = \cO(f(\bar{\alpha}))$.
\end{itemize}

\smallskip

(2)
Assume that $|Q_0|=2$.
Then $(Q,f)$ is of the form
\[
  \xymatrix{
    1
    \ar@(ld,ul)^{\alpha}[]
    \ar@<.5ex>[r]^{\beta}
    & 2
    \ar@<.5ex>[l]^{\gamma}
    \ar@(ru,dr)^{\eta}[]
  } 
\]
with $f$-orbits 
$(\alpha)$, $(\beta\ \eta\ \gamma)$.
Then we have the $g$-orbits 
$\cO(\alpha) = (\alpha\ \beta\ \gamma)$
and
$\cO(\eta) = (\eta)$,
and
$\partial(Q,f) = \{ 1 \}$.
We note that $\bar{\alpha} = \beta$
and $\eta = f(\bar{\alpha})$.
Assume first that $m_{\cO(\alpha)} \geqslant 2$.
If $\eta$ is not virtual 
the $\Lambda$ is not isomorphic 
to a weighted surface algebra, 
by Propositions \ref{prop:5.1} and \ref{prop:5.2}.
Further, if $\eta$ is virtual and
$k = m_{\cO(\alpha)} \geqslant 2$,
then applying 
Proposition \ref{prop:8.1}
(and its proof) we conclude that $\Lambda$
is isomorphic to $Q(2\cA)^k(b)$
with $b \in K^*$.
Then by Proposition \ref{prop:8.2}, 
$\Lambda$ is not isomorphic to a weighted surface algebra. 
Finally if $m_{\mathcal{O}(\alpha)}=1$, 
then $\Lambda$ is isomorphic to a weighted surface algebra, 
this follows from 
Propositions \ref{prop:9.1}, \ref{prop:9.2} and \ref{prop:9.3}.


\end{document}